\documentclass[12pt]{article} 
\usepackage[colorlinks]{hyperref}
\usepackage{graphicx}%
\usepackage{multirow}%
\usepackage{amsmath,amssymb,amsfonts}%
\usepackage{amsthm}%
\usepackage{mathrsfs}%
\usepackage[title]{appendix}%
\usepackage{xcolor}%
\usepackage{textcomp}%
\usepackage{manyfoot}%
\usepackage{booktabs}%
\usepackage{algorithm}%
\usepackage{algorithmicx}%
\usepackage{algpseudocode}%
\usepackage{listings}%
\usepackage[mathscr]{eucal}
\usepackage{bussproofs}
\usepackage{lineno}
\usepackage{quiver}
\usepackage{float}

\usepackage{authblk}

\theoremstyle{plain} 
\newtheorem{theorem}{Theorem}[section]

\newtheorem{lemma}[theorem]{Lemma}

\newtheorem{corollary}[theorem]{Corollary}
\theoremstyle{definition}
\newtheorem{definition}[theorem]{Definition}
\newtheorem{example}[theorem]{Example}
\newtheorem{remark}[theorem]{Remark}

\begin{document}

\title{On a Generalization of Heyting Algebras II}

\author[1]{Amirhossein Akbar Tabatabai}
\author[2]{Majid Alizadeh}
\author[2]{Masoud Memarzadeh}
\affil[1]{Institute of Mathematics, Czech Academy of Sciences}
\affil[2]{School of Mathematics, Statistics and Computer Science, College of Science, University of Tehran}

\date{ }

\maketitle

\begin{abstract}  
A $\nabla$-algebra is a natural generalization of a Heyting algebra, unifying several algebraic structures, including bounded lattices, Heyting algebras, temporal Heyting algebras, and the algebraic representation of dynamic topological systems. In the prequel to this paper \cite{OGHI}, we explored the algebraic properties of various varieties of $\nabla$-algebras, their subdirectly-irreducible and simple elements, their closure under Dedekind-MacNeille completion, and their Kripke-style representation. 

In this sequel, we first introduce $\nabla$-spaces as a common generalization of Priestley and Esakia spaces, through which we develop a duality theory for certain categories of $\nabla$-algebras. Then, we reframe these dualities in terms of spectral spaces and provide an algebraic characterization of natural families of dynamic topological systems over Priestley, Esakia, and spectral spaces. Additionally, we present a ring-theoretic representation for some families of $\nabla$-algebras. Finally, we introduce several logical systems to capture different varieties of $\nabla$-algebras, offering their algebraic, Kripke, topological, and ring-theoretic semantics, and establish a deductive interpolation theorem for some of these systems.\\

\end{abstract}

\section{Introduction}

In many logical frameworks, a binary operation known as \emph{implication} serves to internalize the provability order between propositions.\footnote{For an extensive explanation of what is meant by internalization, see \cite{ImSpace,OGHI}.} Examples of such implications include classical, intuitionistic, and substructural implications, philosophically motivated conditionals \cite{NuteCross}, weak sub-intuitionistic implications \cite{Vi2,visser1981propositional,Ru}, as well as mathematically motivated implications in provability logic \cite{visser1981propositional} and preservability logic \cite{Iem1,Iem2,LitViss}.

In \cite{ImSpace}, the first author introduced an abstract notion of implication that unifies all these instances. This abstract framework was intended to isolate the common behavior of various types of implications, facilitating the study of the additional structure that internalization provides. It also aimed to establish general theorems about implications, particularly regarding negative results. For example, while it is well-known that the locale of open sets in a topological space forms a Heyting algebra, the inverse images of continuous functions do not necessarily preserve the Heyting implication. One might seek a geometric family of implications—one that is stable under inverse images of all continuous functions. However, as demonstrated in \cite{geometric}, no such non-trivial family exists.

One of the well-behaved implications is the Heyting implication which owes its well-behaved behavior to the full adjunction it enjoys. However, due to the lack of residuation, not all implications are mathematically well-behaved, as explained in the introduction of \cite{OGHI}. To introduce a family of well-behaved implications enjoying a sort of residuation, \cite{ImSpace,OGHI} introduced a generalization of a Heyting algebra called a $\nabla$-algebra. Formally, a $\nabla$-algebra is a tuple $(\mathsf{A}, \nabla, \to)$, where $\mathsf{A}$ is a bounded lattice, $\nabla$ and $\to$ are unary and binary operations on $\mathsf{A}$ such that $\nabla c \wedge a \leq b$ iff $c \leq a \to b$, for any $a, b, c \in \mathsf{A}$. We called a $\nabla$-algebra distributive (resp. Heyting) if $\mathsf{A}$ is distributive (resp. Heyting) and it is called normal if $\nabla$ commutes with all finite meets. $\nabla$-algebras are the common generalization of bounded lattices ($\nabla a=0$ and $a \to b=1$) and Heyting algebras ($\nabla a=a$ and $a \to b$ as the Heyting implication). The significance of $\nabla$-algebras was extensively discussed in the prequel to this paper \cite{OGHI}, so we will not repeat that discussion here. However, it is worth briefly mentioning the main points.

First, in \cite{ImSpace,OGHI}, it is shown that $\nabla$-algebras are powerful enough to represent all abstract implications. Therefore, $\nabla$-algebras provide a sufficiently general family of well-behaved implications combining the generality of the abstract implications and the well-behaved nature of Heyting implications. Thanks to this representation theorem, the study of $\nabla$-algebras leads to the study of all possible implications in an indirect manner. Second, defining $\Box a=1 \to a$ in any $\nabla$-algebra, it is easy to see that if $\mathsf{A}$ is a Heyting algebra with the Heyting implication $\supset$, the implication $\to$ is definable by $a \to b =\Box(a \supset b)$. This observation shows that a Heyting $\nabla$-algebra is essentially a Heyting algebra equipped with an adjunction $\nabla \dashv \Box$. One can read $\nabla$ as the existential past modality, i.e., the modality ``it is true at some point in the past" and $\Box$ as the universal future modality, i.e., the modality ``it is always true in the future". Therefore, Heyting $\nabla$-algebras are algebras suitable for a very basic intuitionistic temporal logic \cite{Ewald,de2017constructive}. For more, see the introduction of \cite{OGHI}. Third, recall that a dynamic topological system is a pair $(X, f)$ of a topological space $X$ as the state space of the system and the continuous map $f: X \to X$ as its dynamism \cite{Akin}. To any dynamic topological system $(X, f)$, one can assign the Heyting $\nabla$-algebra $(\mathcal{O}(X), f^{-1}, \to_f)$, where $U \to_f V=f_*(int(U^c \cup V))$ and $f_*: \mathcal{O}(X) \to \mathcal{O}(X)$ is the right adjoint of $f^{-1}: \mathcal{O}(X) \to \mathcal{O}(X)$. The $\nabla$-algebra is normal as $f^{-1}$ preserves all finite intersections. Having this example in mind, one can interpret normal Heyting $\nabla$-algebras as the algebraic and elementary formalization of dynamic topological systems, where the underlying Heyting algebra encodes the locale of the opens in an elementary manner while the adjunction $\nabla \dashv \Box$ captures the adjunction $f^{-1} \dashv f_*$ encoding the continuous dynamism. The approach is similar to \cite{artemov1997modal}, where a classical $\mathsf{S4}$-style modal algebra is used instead of a Heyting algebra to encode the opens of the space. There are also many studies on dynamic topological systems that utilize Heyting algebras as the base structure \cite{Fer}. For further details, refer to the introduction of \cite{OGHI}.

In the first part of this series of two papers \cite{OGHI}, we explored the algebraic properties of $\nabla$-algebras. We introduced several families of $\nabla$-algebras and demonstrated that they form a variety closed under Dedekind-MacNeille completion in many cases. For some of these varieties, we also characterized subdirectly-irreducible and simple elements. Additionally, for distributive $\nabla$-algebras, we provided a Kripke-style representation theorem, which we used to show that certain varieties of normal distributive $\nabla$-algebras possess the amalgamation property.

In this paper, as the second part of the series, we begin with a brief review of the results of \cite{OGHI} in Section \ref{Recall}. Then, in Section \ref{PriEsaDuality}, we focus on duality theory for distributive $\nabla$-algebras. We introduce a common generalization of Priestley and Esakia spaces, termed a $\nabla$-space, and demonstrate that the category of $\nabla$-spaces (with their appropriate maps) is dually equivalent to the category of distributive $\nabla$-algebras. This result unifies and generalizes the Priestley and Esakia duality theorems into a single theorem. We also prove a similar result for certain subcategories of distributive $\nabla$-algebras. 
In particular, to capture the dynamic nature of normal distributive $\nabla$-algebras, we introduce a special family of dynamic Priestley spaces called generalized Esakia spaces. These are pairs $(X, \leq, f)$, where $(X, \leq)$ is a Priestley space and $f: (X, \leq) \to (X, \leq)$ is a Priestley map satisfying an Esakia-style condition: for any clopen $U \subseteq X$, the subset $\downarrow \! f[U]$ is clopen. We show that the category of normal distributive $\nabla$-algebras is dually equivalent to the category of generalized Esakia spaces. This result also leads to an interesting characterization of the category of Heyting $\nabla$-algebras, where $\nabla$ is a bijection, as the dual category of reversible dynamic Esakia spaces—pairs $(X, \leq, f)$ where $(X, \leq)$ is an Esakia space and $f: (X, \leq) \to (X, \leq)$ is an Esakia isomorphism.
Then, in Section \ref{SectionSpectral}, we turn to spectral spaces and provide a similar duality theory for the spectral version of a $\nabla$-space and a generalized Esakia spaces called a $\nabla$-spectral space and generalized H-spectral spaces, respectively. It is known that any Heyting algebra can be represented by the Heyting algebra of the radical ideals of a commutative unital ring. In Section \ref{RingTheoretic}, we use our spectral duality to extend this result and offer a ring-theoretic representation for normal distributive $\nabla$-algebras. This representation theorem provides an order-theoretic look into a slightly modified version of dynamic rings, i.e., rings with a homomorphism over them. Finally, in Section \ref{Logics}, we introduce several logical systems for $\nabla$-algebras. Using the representation theorems presented in \cite{OGHI} and in Secion \ref{RingTheoretic}, we establish algebraic, Kripke, topological and ring-theoretic semantics for these logics. We also prove deductive interpolation theorem for some of these systems. \\

\noindent \textbf{Acknowledgments}

\noindent We wish to thank Nick Bezhanishvili and George Metcalfe for their helpful suggestions. The first author also gratefully acknowledges the support of the FWF project P 33548, the Czech Academy of Sciences (RVO 67985840) and the GA\v{C}R grant 23-04825S. 

\section{A Quick Look into the Prequel} \label{Recall}
In this section, we recall several notions, constructions, and theorems presented in \cite{OGHI}. We assume familiarity with the standard concepts from order theory, universal algebra, and category theory. For the first two, we refer the reader to the comprehensive preliminaries section in \cite{OGHI}, while for the latter, see \cite{Bor1}. Additionally, we will draw on basic concepts from logic and topology, introducing the less familiar concepts as needed.

Let us start with $\nabla$-algebras and some of their interesting families. 
\begin{definition} \cite{OGHI}
Let $\mathsf{A}=(A, \wedge, \vee, 0, 1)$ be a bounded lattice with the lattice order $\leq$. A tuple $\mathcal{A}=(\mathsf{A}, \nabla, \to)$ is called a \emph{$\nabla$-algebra} if $\nabla c \wedge a \leq b$ is equivalent to $c \leq a \to b$, for any $a, b, c \in A$. In any $\nabla$-algebra, $\Box a$ is defined as $1 \to a$. Moreover:
\begin{description}
\item[$(D)$]
If $\mathsf{A}$ is distributive, then the $\nabla$-algebra is called \emph{distributive}.
\item[$(H)$]
If $\mathsf{A}$ is a Heyting algebra, then the $\nabla$-algebra is called \emph{Heyting}. Notice that being Heyting 
does not mean that $\to$ is the Heyting implication itself. It only dictates that the Heyting implication exists over $\mathsf{A}$.
Moreover, notice that the Heyting implication is not included in the signature of a Heyting $\nabla$-algebra. If we add it to the signature to have a structure in the form $(\mathsf{A}, \nabla, \to, \supset)$, where $\supset$ is the Heyting implication over $\mathsf{A}$, we call the structure an ``\emph{explicitly Heyting $\nabla$-algebra}". This difference in the signature is important when we investigate the algebraic or categorical properties of $\nabla$-algebras.
\item[$(N)$]
If $\nabla$ commutes with all finite meets, i.e., $\nabla 1=1$ and $\nabla (a \wedge b)=\nabla a \wedge \nabla b$, for any $a, b \in A$, then the $\nabla$-algebra is called \emph{normal}.
\item[$(R)$]
If $a \leq \nabla a$, for any $a \in A$, the $\nabla$-algebra is called \emph{right}.
\item[$(L)$]
If $\nabla a \leq a$, for any $a \in A$, the $\nabla$-algebra is called \emph{left}.
\item[$(Fa)$]
If $\nabla$ is surjective (equivalently $\Box$ is injective or $\nabla \Box a=a$, for any $a \in A$ \cite{OGHI}), the $\nabla$-algebra is called \emph{faithful}.
\item[$(Fu)$]
If $\Box$ is surjective (equivalently $\nabla$ is injective or $\Box\nabla a=a$, for any $a \in A$ \cite{OGHI}), the $\nabla$-algebra is called \emph{full}.
\end{description}
For any $C \subseteq \{D, H, N, R, L, Fa, Fu\}$, by $\mathcal{V}(C)$, we mean the class of all $\nabla$-algebras with the properties described in the set $C$. For instance, $\mathcal{V}(\{N, D\})$ is the class of all normal distributive $\nabla$-algebras. Sometimes, for simplicity, we omit the brackets. For instance, we write $\mathcal{V}(N, D)$ for $\mathcal{V}(\{N, D\})$ and if $X \in \{D, H, N, R, L, Fa, Fu\}$, we write $\mathcal{V}(X, C)$ for $\mathcal{V}(\{X\} \cup C)$. By $\mathcal{V}_H(C)$, we mean the class of all explicitly Heyting $\nabla$-algebras satisfying the conditions in $C$. For two $\nabla$-algebras $\mathcal{A}=(\mathsf{A}, \nabla_{\mathcal{A}}, \to_{\mathcal{A}})$ and $\mathcal{B}=(\mathsf{B}, \nabla_{\mathcal{B}}, \to_{\mathcal{B}})$, by a \emph{$\nabla$-algebra morphism}, we mean a bounded lattice morphism $f: \mathsf{A} \to \mathsf{B}$ that also preserves $\nabla$ and $\to$, meaning that $f(\nabla_{\mathcal{A}} a)=\nabla_{\mathcal{B}} f(a)$ and $f(a \to_{\mathcal{A}} b)=f(a) \to_{\mathcal{B}} f(b)$, for any $a, b \in A$. If both $\nabla$-algebras are explicitly Heyting and $f$ also preserves the Heyting implication, it is called a \emph{Heyting $\nabla$-algebra morphism}. A $\nabla$-algebra morphism is called an \emph{embedding} if it is injective. For any $C \subseteq \{D, H, N, R, L, Fa, Fu\}$, the class $\mathcal{V}(C)$ of $\nabla$-algebras together with the $\nabla$-algebra morphisms forms a category, denoted by $\mathbf{Alg}_{\nabla}(C)$. Similarly, the class  $\mathcal{V}_H(C)$ of explicitly Heyting $\nabla$-algebras with Heyting morphisms forms a category, denoted by $\mathbf{Alg}^H_{\nabla}(C)$.
\end{definition}

We use capital Greek letters $A$, $B$, ... for sets, the sanserif font as in $\mathsf{A}$, $\mathsf{B}$, ... to denote bounded lattices and calligraphic letters $\mathcal{A}$, $\mathcal{B}$, ... for $\nabla$-algebras. For brevity, when we are working with a $\nabla$-algebra $\mathcal{A}$, we write $\mathsf{A}$ for its underlying bounded lattice and $A$ for its underlying set, without any further explanation.

\begin{example}\cite{OGHI}
Let $\mathsf{A}$ be a bounded lattice. The tuple $(\mathsf{A}, \nabla, \to)$ is a left $\nabla$-algebra, where $\nabla(a)=0$ and $a \to b=1$, for any $a, b \in A$. For another example,
let $\mathsf{A}$ be a Heyting algebra. Then, the tuple $(\mathsf{A}, \nabla, \supset)$ is a normal distributive faithful and full $\nabla$-algebra, where $\nabla(a)=a$, for any $a \in A$ and $\supset$ is the Heyting implication. 
\end{example}

Recall the following notions from topology. A topological space is called $T_0$, if for any two different points $x, y \in X$, there is an open set which contains one of these points and not the other. It is called $T_D$, if for any $x \in X$, there is an open $U$ such that $x \in U$ and $U-\{x\}$ is open. A continuous map is called a \emph{topological embedding} if $f$ induces a homeomorphism between $X$ and $f[X]$.

\begin{theorem}\label{InjSurjforContinuous} \cite{StoneSpaces}
Let $X$ and $Y$ be two topological spaces and $f: X \to Y$ be a continuous map. Then, 
\begin{itemize}
\item[$\bullet$]
If $f$ is surjective, then $f^{-1}: \mathcal{O}(Y) \to \mathcal{O}(X)$ is one-to-one. The converse is true, if $Y$ is $T_D$.
\item[$\bullet$]
If $f$ is a topological embedding, then $f^{-1}: \mathcal{O}(Y) \to \mathcal{O}(X)$ is surjective. The converse is also true, if $X$ is $T_0$.
\end{itemize}
\end{theorem}

\begin{example}\label{DTS}
\emph{(Dynamic Topological Systems \cite{OGHI})}
Let $X$ be a topological space and $\pi: X \to X$ be a continuous function. Then, the pair $(X, \pi)$ is called a \emph{dynamic topological system}. If $\pi$ is a topological embedding (resp. surjection), the dynamic topological system is called \emph{faithful} (resp. \emph{full}).  For any dynamic topological system $(X, \pi)$, define $\to_\pi$ over $\mathcal{O}(X)$ by $U \to_\pi V=\pi_* (int(U^c \cup V))$, where $int$ is the interior operation, $U^c$ is the complement of $U$ and $\pi_*: \mathcal{O}(X) \to \mathcal{O}(X)$ is the right adjoint of $\pi^{-1}$. Then, the structure $\mathcal{T}(X, f)=(\mathcal{O}(X), \pi^{-1}, \to_\pi)$ is a normal Heyting $\nabla$-algebra. Moreover, using Theorem \ref{InjSurjforContinuous}, we can observe that if $(X, \pi)$ is full, i.e., $\pi$ is surjective, then the $\nabla$-algebra $\mathcal{T}(X, f)$ is full, as the surjectivity of $\pi$ implies the injectivity of $\nabla=\pi^{-1}$. The converse also holds if $X$ is $T_D$. In addition, if $(X, \pi)$ is faithful, i.e., $\pi$ is a topological embedding, then the $\nabla$-algebra $\mathcal{T}(X, f)$ is faithful, as being a topological embedding implies the surjectivity of $\nabla=\pi^{-1}$. The converse also holds if $X$ is $T_0$.  
\end{example}

Following \cite{OGHI}, for any $\nabla$-algebra $\mathcal{A}$, define the operations $\nabla$ and $\to$ on its lattice of normal ideals, $\mathcal{N}(\mathsf{A})$, by
$\nabla N = \bigvee_{n \in N} (\nabla n]$ and  $M \to N = \{x \in A \mid \forall m \in M, \nabla x \wedge m \in N \}$. Moreover, consider the map $j: \mathsf{A} \to \mathcal{N}(\mathsf{A})$ defined by $j(a)=\downarrow \! a$. The following theorem shows that the tuple $\mathcal{N}(\mathcal{A})=(\mathcal{N}(\mathsf{A}), \nabla, \to)$ is a $\nabla$-algebra and the lattice embedding $j$ is a $\nabla$-algebra morphism.

\begin{theorem}\label{DedekindCompletionThm} \cite{OGHI}
Let $C \subseteq \{H, N, R, L, Fa, Fu\}$. Then, the classes $\mathcal{V}(C)$ and $\mathcal{V}_H(C)$ are varieties. Moreover, $\mathcal{N}(\mathcal{A}) \in \mathcal{V}(C)$ and the canonical embedding $j: \mathcal{A} \to \mathcal{N}(\mathcal{A})$ is a $\nabla$-algebra embedding, for any $ \mathcal{A} \in \mathcal{V}(C)$. If $H \in C$, the embedding is also a Heyting $\nabla$-algebra morphism. In this sense, $\mathcal{V}(C)$ is closed under the Dedekind-MacNeille completion.
\end{theorem}

The following provides a suitable family of Kripke frames that can be used to construct distributive $\nabla$-algebras. As we will see later in this section, this family is even powerful enough to represent all distributive $\nabla$-algebras.
\begin{definition}\label{DefKripke}
\emph{(Kripke Frames \cite{OGHI})}
Let $(W, \leq)$ be a poset. The tuple $(W, \leq, R)$ is called a \emph{Kripke frame} if $R$ is compatible with $\leq$, i.e., if $k'\leq k$, $(k, l) \in R$ and $l \leq l'$, then $(k', l') \in R$:
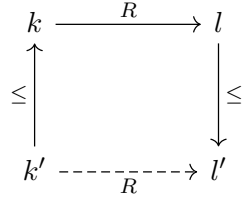
\[\small \begin{tikzcd}
	k && l \\
	\\
	{k'} && {l'}
	\arrow["\leq", from=3-1, to=1-1]
	\arrow["R", from=1-1, to=1-3]
	\arrow["\leq", from=1-3, to=3-3]
	\arrow["R"', dashed, from=3-1, to=3-3]
\end{tikzcd}\]
for any $k,l,k',l'\in W$. Moreover: 
\begin{description}
\item[$(N)$]
if there exists an order-preserving function $\pi : W \to W$, called  the \emph{normality witness}, such that $(x, y) \in R$ iff $x \leq \pi(y)$, then the Kripke frame is called \emph{normal},
\item[$(R)$]
if $R$ is reflexive or equivalently $\leq \, \subseteq R$, then the Kripke frame is called \emph{right},
\item[$(L)$]
if $R \; \subseteq \; \leq$, then the Kripke frame is called \emph{left},
\item[$(Fa)$]
if for any $x \in W$, there exists $y \in W$ such that $(y, x) \in R$ and for any $z \in W$ such that $(y, z) \in R$ we have $x \leq z$, then the Kripke frame is called \emph{faithful},
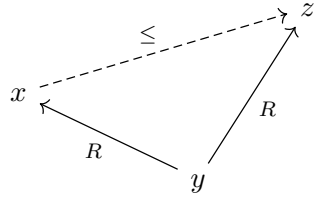
\[\small \begin{tikzcd}
	&&& z \\
	x \\
	&& y
	\arrow["R", from=3-3, to=2-1]
	\arrow["R"', from=3-3, to=1-4]
	\arrow["\leq", dashed, from=2-1, to=1-4]
\end{tikzcd}\]
\item[$(Fu)$]
if for any $x \in W$, there exists $y \in W$ such that $(x, y) \in R$ and for any $z \in W$ such that $(z, y) \in R$ we have $z \leq x$, then the Kripke frame is called \emph{full}.
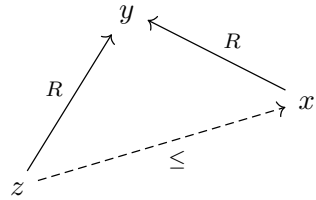
\[\small \begin{tikzcd}
	& y \\
	&&& x \\
	z
	\arrow["R"', from=2-4, to=1-2]
	\arrow["R", from=3-1, to=1-2]
	\arrow["\leq"', dashed, from=3-1, to=2-4]
\end{tikzcd}\]
\end{description}
For any $C \subseteq \{N, R, L, Fa, Fu\}$, by $\mathbf{K}(C)$, we mean the class of all Kripke frames with the properties described in the set $C$. For instance, $\mathbf{K}(\{N, Fa\})$ is the class of all normal faithful Kripke frames. \\
If $\mathcal{K}=(W, \leq, R)$ and $\mathcal{K}'=(W', \leq', R')$ are two Kripke frames, then by a \emph{Kripke morphism} $f: \mathcal{K} \to \mathcal{K}'$, we mean an order-preserving function from $W$ to $W'$ such that:
\begin{itemize}
\item[$\bullet$]
For any $k, l \in W$, if $(k, l) \in R$ then $(f(k), f(l)) \in R'$,
\item[$\bullet$]
for any $k \in W$ and any $l' \in W'$ such that $(f(k), l') \in R'$, there exists $l \in W$ such that $(k, l) \in R$ and $f(l)=l'$,
\item[$\bullet$]
for any $k \in W$ and any $l' \in W'$ such that $(l', f(k)) \in R'$, there exists $l \in W$ such that $(l, k) \in R$ and $f(l) \geq' l'$.
\end{itemize}
If we also have the following condition:
\begin{itemize}
\item[$\bullet$]
for any $k \in W$ and any $l' \in W'$ such that $f(k) \leq' l'$, there exists $l \in W$ such that $k \leq l$ and $f(l)=l'$,
\end{itemize}
then the Kripke morphism $f$ is called a \emph{Heyting  Kripke morphism}. Kripke frames and Kripke morphisms form a category that we loosely denote by its class of objects $\mathbf{K}(C)$. If we use the Heyting Kripke morphisms, instead, then we denote the subcategory by $\mathbf{K}^H(C)$. 
\end{definition}
For normal Kripke frames, \cite{OGHI} provided an alternative presentation for the conditions in $\{R, L, Fa, Fu\}$ in terms of the normality witness:
\begin{lemma}\label{NormalForConditions}\label{NormalityForMorph}
Let $\mathcal{K}=(W, \leq, R)$ be a normal Kripke frame with the normality witness $\pi$. Then: 
\begin{description}
\item[$(i)$]
$(R)$ is equivalent to the condition that $ w \leq \pi(w)$, for any $w \in W$,
\item[$(ii)$]
$(L)$ is equivalent to the condition that $\pi(w) \leq w$, for any $w \in W$,
\item[$(iii)$]
$(Fa)$ is equivalent to the condition that $\pi$ is an order-embedding, i.e., if $\pi (u) \leq \pi(v)$ then $u \leq v$, for any $u, v \in W$,
\item[$(iv)$]
$(Fu)$ is equivalent to the surjectivity of $\pi$.
\end{description}
Moreover, if $\mathcal{K}=(W, \leq, R)$ and $\mathcal{K}'=(W', \leq', R')$ are two normal Kripke frames with the normality witnesses $\pi$ and $\pi'$, respectively, then for an order-preserving map $f: W \to W'$, the following are equivalent:
\begin{description}
\item[$(i)$]
$f$ is a Kripke morphism,
\item[$(ii)$]
$f \circ \pi=\pi' \circ f$ and $\pi'^{-1}(\uparrow \!\!f(k))=f[\pi^{-1}(\uparrow \!\! k)]$, for any $k \in W$.
\end{description}
\end{lemma}
Finally, we present the connection between the categories of distributive $\nabla$-algebras and Kripke frames \cite{OGHI}. First, any Kripke frame gives rise to a $\nabla$-algebra in a canonical way. Define the assignment $\mathfrak{U}$ on the Kripke frame $\mathcal{K}=(W, \leq, R)$ as the tuple $ \mathfrak{U}(\mathcal{K})=(U(W, \leq), \nabla_{\mathcal{K}}, \to_{\mathcal{K}})$, where $U(W, \leq)$ is the bounded distributive lattice of the upsets of $(W, \leq)$ and $\nabla_{\mathcal{K}}(U)=\{x \in X \mid \exists y \in U, (y, x) \in R \}$ and $U \to_{\mathcal{K}} V=\{x \in X \mid \forall y \in U, [(x, y) \in R \Rightarrow y \in V] \}$. Moreover, for any Kripke morphism $f: \mathcal{K} \to \mathcal{K}'$, define $\mathfrak{U}(f)=f^{-1}: \mathfrak{U}(\mathcal{K}') \to \mathfrak{U}(\mathcal{K})$.  
For the converse, let $\mathcal{A}=(\mathsf{A}, \nabla, \to)$ be a distributive $\nabla$-algebra. Define $\mathfrak{P}(\mathcal{A})=(\mathcal{F}_p(\mathsf{A}), \subseteq, R_{\mathcal{A}})$, where $\mathcal{F}_p(\mathsf{A})$ is the set of all prime filters of $\mathsf{A}$ and the relation $R_{\mathcal{A}}$ as the set of pairs
$(P, Q)$ such that [($a \to b \in P$ and $a \in Q$) implies $b \in Q$], for any $a, b \in A$. We have $(P, Q) \in R_{\mathcal{A}}$ iff $\nabla [P] \subseteq Q$ \cite{OGHI}.
Moreover, for any $\nabla$-algebra morphism $f: \mathcal{A} \to \mathcal{B}$ define $\mathfrak{P}(f)=f^{-1}: \mathfrak{P}(\mathcal{B}) \to \mathfrak{P}(\mathcal{A})$ and set $i_{\mathcal{A}}: \mathsf{A} \to U(\mathcal{F}_p(\mathsf{A}), \subseteq)$ as $i_{\mathcal{A}}(a)=\{P \in \mathcal{F}_p(\mathsf{A}) \mid a \in P\}$.
\begin{theorem}\label{KripkeEmbedding} 
\begin{description}
    \item[$(i)$] 
The assignment $\mathfrak{U}: \mathbf{K}^{op} \to \mathbf{Alg}_{\nabla}(H)$ is a functor and for any $C \subseteq \{N, R, L, Fa, Fu\}$, if $\mathcal{K} \in \mathbf{K}(C)$, then $\mathfrak{U}(\mathcal{K})$ lands in $\mathbf{Alg}_{\nabla}(C, H)$. Moreover, the restriction of the functor $\mathfrak{U}$ to $[\mathbf{K}^H(C)]^{op}$ lands in $\mathbf{Alg}^H_{\nabla}(C)$. 
   \item[$(ii)$] 
The assignment $\mathfrak{P}: \mathbf{Alg}_{\nabla}(D) \to \mathbf{K}^{op}$ is a functor and $i_{\mathcal{A}}: \mathcal{A} \to \mathfrak{U} (\mathfrak{P}(\mathcal{A}))$ is a $\nabla$-algebra embedding, natural in $\mathcal{A}$. Moreover, for any $C \subseteq \{N, R, L, Fa, Fu\}$, if $\mathcal{A} \in \mathbf{Alg}_{\nabla}(C, D)$, then $\mathfrak{P}(\mathcal{A})$ lands in $[\mathbf{K}(C)]^{op}$. It also maps $\mathbf{Alg}^H_{\nabla}(C)$ to $[\mathbf{K}^H(C)]^{op}$ where $i_{\mathcal{A}}$ becomes a Heyting $\nabla$-algebra morphism.
\end{description}
\end{theorem}

Finally, let us recall the amalgamation property proved in \cite{OGHI}:

\begin{theorem}\label{Amalgamation}(Amalgamation) Let $C \subseteq \{R, L, Fa\}$. Then, the varieties $\mathcal{V}(C, D, N)$ and $\mathcal{V}_H(C, N)$ have the amalgamation property.
\end{theorem}

\section{Priestley-Esakia Duality for Distributive $\nabla$-algebras}\label{PriEsaDuality}
In this section, we will first introduce a new spatial entity called a \emph{$\nabla$-space} to offer a common generalization of both Priestley and Esakia spaces. Then, leveraging the Kripke representation developed in \cite{OGHI} and recalled in Section \ref{Recall}, we will establish a duality between various categories of distributive $\nabla$-algebras and their corresponding $\nabla$-spaces. This result not only generalizes but also unifies Priestley duality for bounded distributive lattices and Esakia duality for Heyting algebras.

\subsection{Priestley and Esakia dualities}
In this subsection, we will briefly recall the notions and constructions underlying Priestley and Esakia dualities.
A pair $(X, \leq)$ of a topological space and a partial order is called a \textit{Priestley space} if $X$ is compact and for any $x, y \in X$, if $x \nleq y$, there exists a clopen upset $U$ such that $x \in U$ and $y \notin U$. A Priestley space is called an \textit{Esakia space} if the set $\downarrow \! U=\{x \in X \mid \exists y \in U \; x \leq y\}$ is clopen, for any clopen $U$. A continuous and order-preserving function $f: (X, \leq_X) \to (Y, \leq_Y)$ between two Priestley spaces is called a \textit{Priestley map}. In case that both $(X, \leq_X)$ and $(Y, \leq_Y)$ are Esakia spaces and $f[\uparrow \!x]=\uparrow \! f(x)$, for any $x \in X$, the Priestley map $f$ is called an \textit{Esakia map}. Priestley spaces and Priestley maps form a category, which we denote by $\mathbf{Pries}$. The same holds for Esakia spaces and Esakia maps. This category is denoted by $\mathbf{Esakia}$.

The following lemma summarizes some well-known properties of Priestley spaces that we will use later in this section. For further details, see \cite{Pri, Bezh}.
\begin{lemma}\label{PropPriestly} 
Every Priestley space $(X, \leq)$ has the following properties:
\begin{itemize}
\item[$\bullet$]
$X$ is a Hausdorff and zero-dimensional space. The latter means that $X$ has a basis consisting of clopen subsets. In fact, these clopen subsets can be chosen from the intersections of clopen upsets and clopen downsets of $X$. In other words, the family of clopen upsets and clopen downsets of $X$ forms a subbasis for $X$.
\item[$\bullet$]
For any closed subset $F \subseteq X$, both $\uparrow \! F$ and $\downarrow \! F$ are closed. More specifically, $\uparrow \! x$ is closed, for any $x \in X$.
\item[$\bullet$]
Each closed upset (downset) of $X$ is an intersection of clopen upsets (downsets) of $X$.
\item[$\bullet$]
Each open upset (downset) of $X$ is a union of clopen upsets (downsets) of $X$.
\item[$\bullet$]
For each pair of closed subsets $F$ and $G$ of $X$, if $\uparrow \! F \cap \downarrow \! G = \varnothing$, then there exists a clopen upset $U$ such that $F \subseteq U$ and $U \cap G = \varnothing$.
\end{itemize}
\end{lemma}

The Priestley (resp. Esakia) duality defines two functors between the category of bounded distributive lattices (resp. Heyting algebras) and their corresponding maps, denoted by $\mathbf{DLat}$ (resp. $\mathbf{Heyting}$), and the dual of the category $\mathbf{Pries}$ (resp. $\mathbf{Esakia}$). We will recall these two functors. First, we have the functor $\mathfrak{S}_0: \mathbf{DLat} \to \mathbf{Pries}^{op}$ defined on objects by $\mathfrak{S}_0(\mathsf{A})=(\mathcal{F}_p(\mathsf{A}), \subseteq)$, where the topology on $\mathcal{F}_p(\mathsf{A})$ is defined by the basis of the opens in the form $\{P \in \mathcal{F}_p(\mathsf{A}) \mid a \in P \; \text{and} \; b \notin P \}$, for any $a, b \in A$ and defined on the morphism $f: \mathsf{A} \to \mathsf{B}$ by $\mathfrak{S}_0(f)=f^{-1}:(\mathcal{F}_p(\mathsf{B}), \subseteq) \to (\mathcal{F}_p(\mathsf{A}), \subseteq)$. The second functor is $\mathfrak{A}_0: \mathbf{Pries}^{op} \to \mathbf{DLat}$ defined on objects by $\mathfrak{A}_0(X, \leq)=CU(X, \leq)$, where $CU(X, \leq)$ is the lattice of all clopen upsets of $X$ with the inclusion as the order and defined on the morphism $f: (X, \leq_X) \to (Y, \leq_Y)$ by $\mathfrak{A}_0(f)=f^{-1}: CU(Y, \leq_{Y}) \to CU(X, \leq_X)$. These two functors also map the subcategories $\mathbf{Heyting}$ and $\mathbf{Esakia}^{op}$ to each other.

\begin{theorem}(Priestley-Esakia Duality (\cite{Pri, Esakia})\label{PriestleyDuality} 
The functors $\mathfrak{S}_0$ and $\mathfrak{A}_0$ and the following natural isomorphisms:
\begin{itemize}
\item[]
$\alpha : \mathsf{A} \to \mathfrak{A}_0\mathfrak{S}_0(\mathsf{A})$ defined by $\alpha(a)=\{P \in \mathcal{F}_p(\mathsf{A}) \mid a \in P\}$,
\item[]
$\beta: (X, \leq) \to \mathfrak{S}_0\mathfrak{A}_0(X, \leq)$ defined by $\beta(x)=\{U \in CU(X, \leq) \mid x \in U\} $.
\end{itemize}
establish an equivalence between the categories $\mathbf{DLat}$ and $\mathbf{Pries}^{op}$.  The same also holds for $\mathbf{Heyting}$ and $\mathbf{Esakia}^{op}$.
\end{theorem}

Finally, let us provide a concrete characterization for the epic and regular monic maps in the category $\mathbf{Pries}$. For that purpose, first recall the following lemma. 
\begin{lemma}\label{SurjToInj} \cite{OGHI}
Let $\mathsf{A}$ and $\mathsf{B}$ be two bounded distributive lattices and $f: \mathsf{A} \to \mathsf{B}$ be an injective bounded lattice map. Then, the map $f^{-1}: \mathcal{F}_{p}(\mathsf{B)} \to \mathcal{F}_{p}(\mathsf{A)}$ is surjective. 
\end{lemma}
Then, we have the following characterization:
\begin{lemma}\label{PriestleyInjSurj}
Let $(X, \leq_X)$ and $(Y, \leq_Y)$ be Priestley spaces and the function $f: (X, \leq_X) \to (Y, \leq_Y)$ be a Priestley map. Then: 
\begin{description}
\item[$(i)$]
$f$ is surjective iff $f^{-1}: CU(Y, \leq_Y) \to CU(X, \leq_X)$ is one-to-one iff $f$ is an epic map in $\mathbf{Pries}$.
\item[$(ii)$]
$f$ is an order-embedding iff $f^{-1}: CU(Y, \leq_Y) \to CU(X, \leq_X)$ is surjective iff $f$ is a regular monic in $\mathbf{Pries}$.
\end{description}
\end{lemma}
\begin{proof}
For $(i)$, first note that the surjectivity of $f$ implies that $f^{-1}$ is one-to-one on all the subsets of $Y$, including the clopen upsets. For the rest, using Priestley duality, Theorem \ref{PriestleyDuality}, w.l.o.g, we can assume that $(X, \leq_X)=(\mathcal{F}_p(\mathsf{A}), \subseteq)$, $(Y, \leq_Y)=(\mathcal{F}_p(\mathsf{B}), \subseteq)$
and $f=\phi^{-1}$, where $\phi: \mathsf{B} \to \mathsf{A}$ is a $\mathbf{DLat}$ morphism. Note that $f^{-1}$ is isomorphic to $\phi$. Then, it is enough to prove the dual statement, consisting of the following two claims: First, if $\phi$ is one-to-one, then $\phi^{-1}$ is surjective and second, $\phi$ is one-to-one iff $\phi$ is monic in $\mathbf{Pries}$. The first claim is proved in Lemma \ref{SurjToInj}. For the second claim, note that all one-to-one maps are clearly monic. The non-trivial converse is a well-known fact about all algebraic structures and uses the existence of the free structures. For the sake of completeness, we will explain the main points of the proof. Assume that $\phi: \mathsf{B} \to \mathsf{A}$ is monic and $\phi(a)=\phi(b)$, for some $a, b \in B$. Consider the bounded distributive lattice $\mathsf{C}=(\{\varnothing, \{1\}, \{0, 1\}\}, \cap, \cup, \varnothing, \{0, 1\})$ which is the free bounded distributive lattice generated by one element. Set $\psi, \theta : \mathsf{C} \to \mathsf{B}$ as the morphisms mapping $\{1\}$ to $a$ and $b$, respectively. Clearly, $\phi \circ \psi=\phi \circ \theta$. As $\phi$ is monic, we have $\psi=\theta$ which implies $a=b$.

For $(ii)$, we first prove the first equivalence. Assume that the function $f^{-1}: CU(Y, \leq_Y) \to CU(X, \leq_X)$ is surjective. Then, $f$ is clearly an order-embedding, because if $f(x) \leq_Y f(y)$ and $x \nleq_X y$, then there exists a clopen upset $U \subseteq X$ such that $x \in U$ and $y \notin U$. Since $f^{-1}$ is surjective, there exists a clopen upset $V \subseteq Y$ such that $f^{-1}(V)=U$. Therefore, $x \in f^{-1}(V)$ but $y \notin f^{-1}(V)$ which is equivalent to $f(x) \in V$ and $f(y) \notin V$. The last is impossible as $f(x) \leq_Y f(y)$  and $V$ is an upset.
For the converse, assume that $f$ is an order-embedding and $U$ is a clopen upset of $X$. We will provide a clopen upset $V$ of $Y$ such that $f^{-1}(V)=U$. First, note that as $U$ and $U^c$ are closed, they are also compact and hence $f[U]$ and $f[U^c]$ are compact and hence closed. We first claim that $\uparrow \!\! f[U] \cap \downarrow \!\! f[U^c]=\varnothing$. Because, if $y \in \; \uparrow \! f[U] \cap \downarrow \! f[U^c]$, there are $x \in U$ and $z \notin U$ such that $f(x) \leq_Y y \leq_Y f(z)$. Since $f$ is an order-embedding, we have $x \leq_X z$ and since $x \in U$ and $U$ is an upset, we reach $z \in U$ which is a contradiction. Hence, $\uparrow \! f[U] \cap \downarrow \! f[U^c]=\varnothing$. By Lemma \ref{PropPriestly}, there exists a clopen upset $V$ such that $f[U] \subseteq V$ and $f[U^c] \cap V=\varnothing$. The former implies $U \subseteq f^{-1}(V)$ and the latter proves $f^{-1}(V) \subseteq U$. Hence, $f^{-1}(V)=U$. 

For the second equivalence, again by Priestley duality, Theorem \ref{PriestleyDuality}, it is enough to prove that  $\phi$ is surjective iff it is regular epic in $\mathbf{DLat}$, for any $\mathbf{DLat}$ morphism $\phi: \mathsf{B} \to \mathsf{A}$. This is again a well-known fact for the algebraic structures. However, to keep the categorical preliminaries as low as possible, we highlight the main points of the proof. First, assume that $\phi$ is a coequalizer of $\psi, \theta: \mathsf{C} \to \mathsf{B}$ in $\mathbf{DLat}$. Consider $\phi[B]$ as the image of $B$ under $\phi$. As $\phi[B]$ is closed under finite meets and joins, it inherits the bounded lattice structure from $\mathsf{A}$. Hence, the inclusion map $i: \phi[B] \to \mathsf{A}$ is a $\mathbf{DLat}$ map. Moreover, the map $\phi': \mathsf{B} \to \phi[B]$ induced by $\phi$ is also a $\mathbf{DLat}$ map. Note that the latter equalizes $\psi$ and $\theta$. Hence, by the universality of the coequalizer, there is a map $j: \mathsf{A} \to \phi[B]$ such that $j\phi=\phi'$:
\[\small\begin{tikzcd}[ampersand replacement=\&]
	{\mathsf{C}} \&\& {\mathsf{B}} \&\& {\mathsf{A}} \\
	\\
	\&\&\&\& {\phi[B]} \\
	\\
	\&\&\&\& {\mathsf{A}}
	\arrow["\psi", shift left, from=1-1, to=1-3]
	\arrow["\theta"', shift right, from=1-1, to=1-3]
	\arrow["\phi", from=1-3, to=1-5]
	\arrow["{\phi'}", from=1-3, to=3-5]
	\arrow["j", from=1-5, to=3-5]
	\arrow["i", from=3-5, to=5-5]
	\arrow["{id_{\mathsf{A}}}", curve={height=-30pt}, from=1-5, to=5-5]
	\arrow["\phi"', from=1-3, to=5-5]
\end{tikzcd}\]
It is clear that $i\phi'=\phi$. Hence, $ij\phi=\phi$. As $id_{\mathsf{A}}\phi=\phi$, by the universality, we have $ij=id_{\mathsf{A}}$ which implies $j(a)=a$, for any $a \in \mathsf{A}$. Hence, for any $a \in A$, as $j(a) \in \phi[B]$, we have $a \in \phi[B]$. Therefore, $\phi[B]=A$ which means that $\phi$ is surjective. For the converse, assume that $\phi: \mathsf{B} \to \mathsf{A}$ is surjective. Consider the poset $\mathsf{C}$ with the underlying set $\{(b, c) \in B^2 \mid \phi(b)=\phi(c)\}$ and the pointwise order. It is easy to see that $\mathsf{C}$ is a bounded distributive lattice. Now, consider the projection functions $p_0, p_1: \mathsf{C} \to \mathsf{B}$. It is clear that both of these functions are $\mathbf{DLat}$ maps coequalized by
$\phi$. To prove the universality, assume that $\psi: \mathsf{B} \to \mathsf{D}$ coequalizes $p_0$ and $p_1$:
\[ \begin{tikzcd}
	{\mathsf{C}} && {\mathsf{B}} && {\mathsf{A}} \\
	\\
	&&&& {\mathsf{D}}
	\arrow["{p_0}", shift left=1, from=1-1, to=1-3]
	\arrow["{p_1}"', shift right=1, from=1-1, to=1-3]
	\arrow["\phi", from=1-3, to=1-5]
	\arrow["\psi"', from=1-3, to=3-5]
	\arrow["j", dashed, from=1-5, to=3-5]
\end{tikzcd}\]
As $\phi$ is surjective, for any $a \in A$, there is a $b \in B$ such that $\phi(b)=a$. Define $j: \mathsf{A} \to \mathsf{D}$ by $j(a)=\psi(b)$, where $\phi(b)=a$. As $\psi$ coequalizes $p_0$ and $p_1$, it is easy to see that $j(a)$ is independent of the choice of $b$. By definition, $j \phi=\psi$. Showing that $j$ is a $\mathbf{DLat}$ morphism and it is unique with the property $j\phi=\psi$ is easy.
\end{proof}

\subsection{$\nabla$-Spaces and generalized Esakia spaces}\label{SubsubsectionofGEsakia}
To make the Kripke representation theorem for distributive $\nabla$-algebras into a full-scale duality, we must refine a Kripke frame by adding a suitable topological structure to specify the upsets coming from the algebraic side. This is the task of the present subsection, where we introduce two new notions of a $\nabla$-space and a generalized Esakia space as the topological versions of a Kripke frame and a normal Kripke frame, respectively.
 
\begin{definition}\label{DefNablaSpace}
A \emph{$\nabla$-space} is a tuple $(X, \leq, R)$ of a Priestley space $(X, \leq)$ and a binary relation $R$ on $X$ such that:
\begin{itemize}
\item[$\bullet$]
$R$ is compatible with the order, i.e., $x'\leq x$, $(x, y) \in R$ and $y \leq y'$ imply $(x', y') \in R$, for any $x, y, x', y' \in X$,
\item[$\bullet$]
$R[x]=\{y \in X \mid (x, y) \in R\}$ is closed, for every $x \in X$,
\item[$\bullet$]
$\lozenge_R(U)=\{x \in X \mid \exists y \in U \; (x, y) \in R\}$ is clopen, for any clopen $U$,
\item[$\bullet$]
$\nabla_R(V)=\{x \in X \mid \exists y \in V \; (y, x) \in R \}$ is a clopen upset, for any clopen upset $V$.
\end{itemize}
Note that any $\nabla$-space is a Kripke frame, if we forget the topology of the space. A $\nabla$-space satisfies a condition in the set $\{N, R, L, Fa, Fu\}$, if it satisfies the condition as a Kripke frame. A $\nabla$-space is called \emph{Heyting}, if $(X, \leq)$ is an Esakia space. 
For $\nabla$-spaces $(X, \leq_X, R_X)$ and $(Y, \leq_Y, R_Y)$, by a \emph{$\nabla$-space map} $f: (X, \leq_X, R_X) \to (Y, \leq_Y, R_Y)$, we mean a Kripke morphism that is also continuous. Note that any $\nabla$-space map is also a Priestley map. A $\nabla$-space map is called \emph{Heyting}, if it is Heyting as a Kripke morphism or equivalently if it is an Esakia map. For any $C \subseteq \{N, H, R, L, Fa, Fu\}$, the class of all $\nabla$-spaces satisfying the conditions in $C$ together with $\nabla$-space maps form a category denoted by $\mathbf{Space}_{\nabla}(C)$. If we restrict the objects to Heyting $\nabla$-spaces and the morphisms to Heyting $\nabla$-space maps, we denote the subcategory by $\mathbf{Space}^{H}_{\nabla}(C)$.
\end{definition}

\begin{example}\label{ExamplesofNablaSpaces}
There are two interesting degenerate families of $\nabla$-spaces corresponding to $R=\varnothing$ and $R=\; \leq$. The first family characterizes Priestley spaces while the second identifies Esakia spaces. To explain, for the former, let $(X, \leq)$ be a Priestley space and set $R=\varnothing$. It is clear that $(X, \leq, R)$ satisfies all the required properties of Definition \ref{DefNablaSpace} as $R[x]=\Diamond_R(U)=\nabla_R(V)=\varnothing$, for any $x \in X$, any clopen $U$ and any clopen upset $V$. Therefore, we can conclude that $\nabla$-spaces generalize Priestley spaces. Moreover, note that if $f:(X, \leq_X) \to (Y, \leq_Y)$ is a Priestley map, then it is also a $\nabla$-space map from $(X, \leq_X, \varnothing)$ to $(Y, \leq_Y, \varnothing)$. Hence, $\mathbf{Pries}$ is a full subcategory of $\mathbf{Space}_{\nabla}$. 
For the second degenerate family, let $(X, \leq)$ be a Priestley space and set $R=\; \leq$. We show that $(X, \leq, R)$ is a $\nabla$-space iff $(X, \leq)$ is an Esakia space. Checking the required properties of Definition \ref{DefNablaSpace}, we can see that all are automatic except one. The compatibility condition is a consequence of transitivity of $\leq$. 
For the second condition, note that $R[x]=\uparrow \! x$ and $\uparrow \! x$ is always closed in any Priestley space, by Lemma \ref{PropPriestly}. For the fourth condition, notice that $\nabla_R(V)=\uparrow \! V$ and as $V$ is an upset, we have $\nabla_R(V)=V$ which makes the last condition automatic, as well. The only non-automatic condition is the third condition. As $\lozenge_R(U)=\downarrow \! U$, the third condition is nothing but the Esakia condition. Therefore, we can conclude that $\nabla$-spaces generalize Esakia spaces, as well. Moreover, note that if $(X, \leq_X) $ and $(Y, \leq_Y)$ are Esakia spaces, then a Priestley map $f:(X, \leq_X) \to (Y, \leq_Y)$ is a $\nabla$-space map from $(X, \leq_X, \leq_X)$ to $(Y, \leq_Y, \leq_Y)$ iff it is an Esakia map. The reason again is that all the properties of a $\nabla$-space map are automatic except one. Preserving the relation $R$ is a result of being order-preserving and the last condition states that if $y \leq_Y f(x)$, there is a $z \in X$ such that $z \leq_X x$ and $y \leq_Y f(z)$ which is clear by setting $z=x$. The only non-trivial condition is the second condition which becomes the Esakia condition. Altogether, we can conclude that $\mathbf{Esakia}$ is also a full subcategory of $\mathbf{Space}_{\nabla}$. 
\end{example}

By a \emph{dynamic object} in a category $\mathcal{C}$, we mean a pair in the form $(A, \pi)$, where $A$ and $\pi: A \to A$ are an object and a morphism in $\mathcal{C}$, respectively. The object $A$ and the morphism $\pi: A \to A$ are called the \emph{space} and the \emph{dynamism} of the dynamic object $(A, \pi)$, respectively. A dynamic object $(A, \pi)$ in $\mathcal{C}$ is called \emph{reversible} if $\pi$ is an isomorphism. By a morphism $f: (A, \pi) \to (B, \tau)$ of dynamic objects in $\mathcal{C}$, we mean a morphism $f: A \to B$ such that $f\pi=\tau f$. 
Dynamic objects in $\mathcal{C}$ with their morphisms form a category denoted by $\mathcal{C}^{\mathbb{N}}$. The full subcategory of reversible dynamic objects in $\mathcal{C}$ is denoted by $\mathcal{C}^{\mathbb{Z}}$. In this section, we are interested in dynamic objects in $\mathbf{Pries}$ and $\mathbf{Esakia}$, also called \emph{dynamic Priestley spaces} and \emph{dynamic Esakia spaces}, respectively.

In the rest of this subsection, we will provide an alternative presentation for the categories of normal (faithful or full) $\nabla$-spaces as some natural and (not necessarily full) subcategories of dynamic Priestley spaces. The characterization also implies that the category $\mathbf{Space}_{\nabla}(N, Fa, Fu)$ is isomorphic to the category of reversible Esakia spaces.
To start, let us rewrite the definition of a normal $\nabla$-space, presented in Definition \ref{DefNablaSpace}, in terms of its normality witness.

\begin{lemma}\label{NormalityForPriestly}
Let $(X, \leq)$ be a Priestley space and $\pi: X \to X$ be a function. Define the relation $R=\{(x, y) \in X^2 \mid x \leq \pi(y)\}$. Then $(X, \leq, R)$ is a normal $\nabla$-space iff $\pi$ is a Priestley map and $\downarrow \! \pi[U]$ is clopen, for any clopen $U$.
\end{lemma}
\begin{proof}
First, note that as $(x, y) \in R$ is equivalent to $x \leq \pi(y)$, we have the following equlities: $R[x]=\pi^{-1}(\uparrow \! x)$, $\lozenge_R(U)=\downarrow \! \pi[U]$ and  $\nabla_R(V)=\pi^{-1}(V)$, for any $x \in X$, any subset $U$ and any upset $V$. Now, assume that $(X, \leq, R)$ is a normal $\nabla$-space. We prove that $\pi$ is a Priestley map and $\downarrow \! \pi[U]$ is clopen, for any clopen $U$. As $(X, \leq, R)$ is normal, there exists an order-preserving function $\sigma: (X, \leq) \to (X, \leq)$ such that $(x, y) \in R$ iff $x \leq \sigma(y)$. Therefore, $x \leq \sigma(y)$ iff $x \leq \pi(y)$ which implies $\sigma=\pi$. Hence, $\pi$ is order-preserving. As for any clopen $U$, the subset $\lozenge_R(U)$ is clopen and we have $\lozenge_R(U)=\downarrow \! \pi[U]$, the only thing to prove is the continuity of $\pi$.  For that purpose, we have to show that $\pi^{-1}(W)$ is open, for any open $W \subseteq X$. As the space is Priestley, it has a sub-basis constituting of clopen upsets and clopen downsets, by Lemma \ref{PropPriestly}. Since $\pi^{-1}$ commutes with union, intersection, and complement, it is enough to show that $\pi^{-1}(V)$ is open, for any clopen upset $V$. But $\pi^{-1}(V)= \nabla_R[V] $ and $\nabla_R[V]$ is a clopen upset. Conversely, assume that $\pi$ is a Priestley map and $\downarrow \! \pi[U]$ is clopen, for any clopen $U$. To show that $(X, \leq, R)$ is a $\nabla$-space, first note that $R$ is compatible with the order as $\pi$ is order-preserving. Secondly, by Lemma \ref{PropPriestly}, the subset $\uparrow \! \! x$ is closed and since $\pi$ is continuous, the set $R[x]=\pi^{-1}(\uparrow \!\! x)$ is also closed. Thirdly, as $\downarrow \!\! \pi[U]=\lozenge_R(U)$, for any clopen $U$, the subset $\lozenge_R(U)$ is clopen. Fourthly, for any clopen upset $V$, we have $\nabla_R(V)=\pi^{-1}(V)$ which is also a clopen upset, by continuity of $\pi$ and the fact that $\pi$ is order-preserving. Note that, for normality, as $R$ is defined by $\{(x, y) \in X^2 \mid x \leq \pi(y) \}$, there is nothing to prove.
\end{proof}

Lemma \ref{NormalityForPriestly} shows that a normal  $\nabla$-space $(X, \leq, R)$ can be equivalently presented by the tuple $(X, \leq, \pi)$, where $(X, \leq)$ is a Priestley  space and $\pi: (X, \leq) \to (X, \leq)$ is a Priestley map such that $\downarrow \! \pi[U]$ is clopen, for any clopen $U$. The first part of the data is a dynamic Priestley space. The second part, however, is reminiscent of the additional condition on Esakia spaces. Therefore, one may argue that a normal $\nabla$-space is essentially a dynamic Priestley space whose dynamism satisfies an \emph{Esakia-style condition}. This new notion deserves to be called a \emph{generalized Esakia space}.

\begin{definition}\label{gEsakia}
A \emph{generalized Esakia space} is a tuple $(X, \leq, \pi)$ where $(X, \leq)$ is a Priestley space and $\pi: (X, \leq) \to (X, \leq)$ is a Priestley map such that $\downarrow \! \pi[U]$ is clopen, for any clopen $U$. It is called \emph{Heyting} if $(X, \leq)$ is an Esakia space. By a \emph{generalized Esakia map} $f: (X, \leq_X, \pi_X) \to (Y, \leq_Y, \pi_Y)$, we mean a Priestley map $f: (X, \leq_X) \to (Y, \leq_Y)$ such that $f\pi_X=\pi_Y f$ and $\pi_Y^{-1}(\uparrow \! f(x))=f[\pi_X^{-1}(\uparrow \! x)]$, for any $x \in X$. It is called \emph{Heyting}, if it is also an Esakia map. Generalized Esakia spaces and generalized Esakia maps form a category, denoted by $\mathbf{gEsakia}$. If we restrict the objects and the morphisms to Heyting objects and Heyting maps, the subcategory is denoted by $\mathbf{gEsakia}^H$.
\end{definition}

\begin{example}\label{ExampleofNormalNablaSpace}
For a Priestley space $(X, \leq)$, it is clear that $(X, \leq, id_X)$ is a generalized Esakia space iff $(X, \leq)$ is an Esakia space. This is another presentation of the second degenerate case in Example \ref{ExamplesofNablaSpaces}, where $R= \; \leq $. More generally, assume that $(X, \leq)$ is a Priestley space and $\pi: (X, \leq) \to (X, \leq)$ is an isomorphism in $\mathbf{Pries}$. Then, the tuple $(X, \leq, \pi)$ is a generalized Esakia space iff $(X, \leq)$ is an Esakia space. The reason simply is that for a homeomorphism $\pi$, the condition ``$\downarrow \! \pi[U]$ is clopen, for any clopen $U$" is equivalent to ``$\downarrow \! U$ is clopen, for any clopen $U$".
\end{example}

So far, we have seen that generalized Esakia spaces and normal $\nabla$-spaces are different presentations of the same mathematical entity. We claim that a similar thing also holds for their corresponding morphisms.
To prove this claim, note that if $(X, \leq_X, R_X)$ and $(Y, \leq_Y, R_Y)$ are two normal $\nabla$-spaces with the normality witnesses $\pi_X$ and $\pi_Y$, respectively, then by Lemma \ref{NormalityForMorph}, a function $f: X \to Y$ is a $\nabla$-space map iff it is a Priestley map,
$f\pi_X=\pi_{Y}f$ and $\pi_Y^{-1}(\uparrow \!\! f(x))=f[\pi_X^{-1}(\uparrow \!\! x)]$, for any $x \in X$. The latter is simply the description of a generalized Esakia map. Putting Lemma \ref{NormalityForPriestly} and the above observation together, we can finally conclude
$\mathbf{Space}_{\nabla}(N) \cong \mathbf{gEsakia}$. Moreover, the isomorphism can be restricted to the Heyting case, i.e., $\mathbf{Spec}^H_{\nabla}(N) \cong \mathbf{gEsakia}^H$. The reason is that a $\nabla$-space $(X, \leq, R)$ is Heyting iff $(X, \leq)$ is Esakia and a $\nabla$-space map is Heyting iff it is an Esakia map, both by definition.
We record these two observations in a corollary.

\begin{corollary}\label{CorollaryforNormalNablaI}
$\mathbf{Space}_{\nabla}(N) \cong \mathbf{gEsakia}$ and $\mathbf{Spec}^H_{\nabla}(N) \cong \mathbf{gEsakia}^H$.
\end{corollary}

To continue our dynamical representation of normal $\nabla$-spaces, let us consider the faithfulness and fullness conditions, as well:

\begin{lemma}\label{PriestleyFullFaithful}
Let $(X, \leq)$ be a Priestley space and $\pi: X \to X$ be a Priestley map such that $\downarrow \! \pi[U]$ is clopen, for any clopen $U$. Define the relation $R=\{(x, y) \in X^2 \mid x \leq \pi(y)\}$. Then:
\begin{description}
\item[$(i)$]
$(X, \leq, R)$ is a normal and faithful $\nabla$-space iff $\pi$ is a Priestley map that is also an order-embedding iff $\pi$ is a regular monic in $\mathbf{Pries}$.
\item[$(ii)$]
$(X, \leq, R)$ is a normal and full $\nabla$-space iff $\pi$ is a  surjective Priestley map iff $\pi$ is an epic map in $\mathbf{Pries}$.
\item[$(iii)$]
$(X, \leq, R)$ is a normal, faithful and full $\nabla$-space iff $\pi$ is an isomorphism in $\mathbf{Pries}$.
\end{description}
\end{lemma}
\begin{proof}
For $(i)$ and $(ii)$, by Lemma \ref{NormalityForPriestly}, the tuple $(X, \leq, R)$ is a normal $\nabla$-space. The equivalences are the consequence of Lemma \ref{NormalForConditions} and Lemma \ref{PriestleyInjSurj}. The part $(iii)$ is a consequence of $(i)$ and $(ii)$.
\end{proof}


\begin{definition}\label{gEsakiaII}
A generalized Esakia space $(X, \leq, \pi)$ is called \emph{regular monic (epic)}, if $\pi$ is regular monic (epic) in $\mathbf{Pries}$ or equivalently order-embedding (resp. surjective), by Lemma \ref{PriestleyInjSurj}. For any $I \subseteq \{rm, e\}$, the full subcategories of generalized Esakia spaces satisfying the conditions in $I$ is denoted by $\mathbf{gEsakia}_{I}$. If we restrict the objects and the morphisms to Heyting objects and Heyting maps, the subcategory is denoted by $\mathbf{gEsakia}_I^H$.
\end{definition}

Similar to our above observation, we have the following isomorphism: 
\begin{corollary}\label{CorollaryforNormalNablaII}
We have the following isomorphisms:
\begin{description}
    \item[$\bullet$]
$\mathbf{Space}_{\nabla}(N, Fa) \cong \mathbf{gEsakia}_{rm}$ and $\mathbf{Space}^H_{\nabla}(N, Fa) \cong \mathbf{gEsakia}_{rm}^H$.
    \item[$\bullet$]
$\mathbf{Space}_{\nabla}(N, Fu) \cong \mathbf{gEsakia}_{e}$ and $\mathbf{Space}^H_{\nabla}(N, Fu) \cong \mathbf{gEsakia}_{e}^H$.
    \item[$\bullet$]
$\mathbf{Space}_{\nabla}(N, Fa, Fu) \cong \mathbf{gEsakia}_{rm, e}$ and $\mathbf{Space}^H_{\nabla}(N, Fa, Fu) \cong \mathbf{gEsakia}_{rm, e}^H$.   
\end{description}
\end{corollary}

Having both faithfulness and fullness conditions on normal $\nabla$-spaces is interesting as it results in reversible dynamic Esakia spaces.

\begin{corollary}\label{CorollaryforNormalNablaIII}
$\mathbf{Space}^H_{\nabla}(N, Fa, Fu) \cong \mathbf{Esakia}^{\mathbb{Z}}$.
\end{corollary}
\begin{proof}
By Corollary \ref{CorollaryforNormalNablaII}, part $(iii)$, the category $\mathbf{Space}^H_{\nabla}(N, Fa, Fu)$ is isomorphic to the category $\mathbf{gEsakia}^H_{rm, e}$. Therefore, it is enough to show the isomorphism between $\mathbf{gEsakia}^H_{rm, e}$ and $\mathbf{Esakia}^{\mathbb{Z}}$.
On the level of objects, we first claim that any generalized Esakia space $(X, \leq, \pi)$, where $\pi:(X, \leq) \to (X, \leq)$ is an isomorphism in $\mathbf{Pries}$ is nothing but a reversible dynamic Esakia space. One direction is clear by definition. For the other, let $(X, \leq)$ be an Esakia space and $\pi: (X, \leq) \to (X, \leq)$ be an isomorphism. Then, by Example \ref{ExampleofNormalNablaSpace}, $(X, \leq, \pi)$ is a generalized Esakia space. 
For the equivalence between the morphisms of $\mathbf{gEsakia}^H_{rm, e}$ and $\mathbf{Esakia}^{\mathbb{Z}}$, the only thing to show is that for any  $f: (X, \leq_X, \pi_X) \to (Y, \leq_Y, \pi_Y)$ as a map between reversible dynamic Esakia spaces, if both $\pi_X: (X, \leq_X) \to (X, \leq_X)$ and $\pi_Y: (Y, \leq_Y) \to (Y, \leq_Y)$ are isomorphisms in $\mathbf{Pries}$ or equivalently in $\mathbf{Esakia}$, then $f$ is also a map between generalized Esakia spaces. This means that we have to show $\pi_Y^{-1}(\uparrow \! f(x))=f[\pi_X^{-1}(\uparrow \! x)]$, for any $x \in X$. To prove, first, note that $f$ is a map between dynamic Esakia spaces. Thus, $f\pi_X=\pi_{Y}f$ which implies $\pi_{Y}^{-1}f=f\pi_X^{-1}$, as $\pi_X$ and $\pi_Y$ are isomorphisms over $(X, \leq_X)$ and $(Y, \leq_Y)$, respectively. Then, consider
\[
\pi_Y^{-1}(\uparrow \! f(x))=\pi_Y^{-1}(f[\uparrow \! x])=f[\pi_X^{-1}(\uparrow x)].
\]
The first equality holds as $f$ is an Esakia map and hence $\uparrow \! f(x)=f[\uparrow \!x]$ and the second is a consequence of $\pi_{Y}^{-1}f=f\pi_X^{-1}$.
\end{proof}

\subsection{Priestley-Esakia duality}

In this subsection, we will provide the promised duality between the categories $\mathbf{Alg}_{\nabla}(D, C)$ and $\mathbf{Space}_{\nabla}(C)$, for any $C \subseteq \{N, H, R, L, Fa, Fu\}$. For that purpose, we need to extend the functors $\mathfrak{S}_0$ and $\mathfrak{A}_0$ from bounded distributive lattices and Priestley spaces to distributive $\nabla$-algebras and $\nabla$-spaces, respectively. To avoid confusion, though, we will denote these new functors by the new names $\mathfrak{S}$ and $\mathfrak{A}$.

\begin{lemma}\label{DefA}
Let $\mathscr{X}=(X, \leq, R)$ be a $\nabla$-space. Define $\nabla_{R}(U)=\{x \in X \mid \exists y \in U, (y, x) \in R \}$ and $U \to_{R} V=\{x \in X \mid  R[x] \cap U \subseteq V \}$ over $CU(X, \leq)$. Then the tuple
$
\mathfrak{A}(\mathscr{X})=(CU(X, \leq), \nabla_{R}, \to_{R})
$
is a $\nabla$-algebra. If we also define $\mathfrak{A}(f)=f^{-1}: \mathfrak{A}(\mathscr{Y}) \to \mathfrak{A}(\mathscr{X})$, for any $\nabla$-space map $f: \mathscr{X} \to \mathscr{Y}$, then $\mathfrak{A}$ defines a functor from $[\mathbf{Space}_{\nabla}]^{op}$ to $\mathbf{Alg}_{\nabla}(D)$. Moreover, for any $C \subseteq \{N, H, R, L, Fa, Fu\}$, the functor $\mathfrak{A}$ maps  $[\mathbf{Space}_{\nabla}(C)]^{op}$ to $\mathbf{Alg}_{\nabla}(D, C)$. The same also holds for $[\mathbf{Space}_{\nabla}^H(C)]^{op}$ and $\mathbf{Alg}^H_{\nabla}(D, C)$
\end{lemma}
\begin{proof}
First, note that $\nabla_{R}$ and $\to_{R}$ are well-defined operations over the set $CU(X, \leq)$. For $\nabla_{R}$, by definition of a $\nabla$-space, $\nabla_{R}$ maps clopen upsets to clopen upsets and hence there is nothing to prove. For $\to_{R}$, note that if $U$ and $V$ are clopen upsets, $U \cap V^c$ is also clopen. By definition of a $\nabla$-space, $\lozenge_{R}(U \cap V^c)$ is clopen. As, $(U \to_{R} V)^c=\lozenge_{R}(U \cap V^c)$, the set $(U \to_{R} V)^c$ and hence $U \to_{R} V$ are both clopen. The set $U \to_{R} V$ is also upward-closed, because if $x \leq y$ and $x \in U \to_{R} V$, then $R[x] \cap U \subseteq V$. Since $R$ is compatible with the order, we have $R[y] \subseteq R[x]$. Hence, $R[y] \cap U \subseteq V$ which implies $y \in U \to_{R} V$. Now, note that $\mathfrak{A}(\mathscr{X})$ is a subset of $\mathfrak{U}(\mathscr{X})$, restricting from all upsets of $(X, \leq)$ to just clopen upsets, with the same algebraic operations. We just proved the closure under $\nabla_R$ and $\to_R$. For the meets and joins, it is an easy consequence of the fact that clopen subsets are closed under finite union and intersections. By Theorem \ref{KripkeEmbedding}, $\mathfrak{U}(\mathscr{X})$ is a $\nabla$-algebra. Hence, $\mathfrak{A}(\mathscr{X})$  inherits the adjunction property $\nabla(-) \wedge U \dashv U \to (-)$, for any clopen upset $U \subseteq X$ which implies that $(CU(X, \leq), \nabla_{R}, \to_{R})$ is a $\nabla$-algebra itself. For the conditions in $C$, if $(X, \leq, R) \in \mathbf{Space}_{\nabla}(C)$, by Theorem \ref{KripkeEmbedding}, $\mathfrak{U}(\mathscr{X}) \in \mathbf{Alg}_{\nabla}(D, C)$. As all conditions in $C$ are universal conditions, they are inherited from $\mathfrak{U}(\mathscr{X})$ to $\mathfrak{A}(\mathscr{X})$. For the morphisms, if $f: (X, \leq_X, R_X) \to (Y, \leq_Y, R_Y)$ is a $\nabla$-space map, then it is by definition a Kripke morphism and hence $\mathfrak{U}(f)=f^{-1}: \mathfrak{U}(Y, \leq_Y, R_Y) \to \mathfrak{U}(X, \leq_X, R_X)$ is a $\nabla$-algebra morphism, again by Theorem \ref{KripkeEmbedding}. For any clopen upset $U \subseteq Y$, the set $f^{-1}(U) \subseteq X$ is also a clopen upset, by the continuity and the monotonicity of $f$. Therefore, as $\mathfrak{A}(f)=f^{-1} :  \mathfrak{A}(Y, \leq_Y, R_Y) \to \mathfrak{A}(X, \leq_X, R_X)$ is a restriction of $\mathfrak{U}(f)$ from all upsets of $(Y, \leq_Y)$ to just clopen upsets, it must be a $\nabla$-algebra morphism, as well. For the Heyting case, if $(X, \leq, R)$ is a Heyting $\nabla$-space, $(X, \leq)$ is an Esakia space and hence $CU(X, \leq)$ is a Heyting algebra by Esakia duality, Theorem \ref{PriestleyDuality}. Moreover, if $f: (X, \leq_X, R_X) \to (Y, \leq_Y, R_Y)$ is a Heyting $\nabla$-space map, $f^{-1}$ preserves the Heyting implication, by Esakia duality, Theorem \ref{PriestleyDuality}.
\end{proof}

\begin{lemma}\label{DefS}
Let $\mathcal{A}$ be a distributive $\nabla$-algebra. Define $\mathfrak{S}(\mathcal{A})=(\mathcal{F}_p(\mathsf{A}), \subseteq, R)$, where $\mathcal{F}_p(\mathsf{A})$ equipped with the Priestley topology as defined by the basis $\{P \in \mathcal{F}_p(\mathsf{A}) \mid a \in P \; \text{and} \; b \notin P \}$, for any $a, b \in A$ and $(P, Q) \in R$ iff $\nabla[P] \subseteq Q$. Then, $\mathfrak{S}(\mathcal{A})$
is a $\nabla$-space. If we also define $\mathfrak{S}(f)=f^{-1}: \mathfrak{S}(\mathcal{B}) \to \mathfrak{S}(\mathcal{A})$, for any $\nabla$-algebra map $f: \mathcal{A} \to \mathcal{B}$, then, $\mathfrak{S}$ defines a functor from $\mathbf{Alg}_{\nabla}(D)$ to $[\mathbf{Space}_{\nabla}]^{op}$. Moreover, for any $C \subseteq \{N, H, R, L, Fa, Fu\}$, the functor $\mathfrak{A}$ maps $\mathbf{Alg}_{\nabla}(D, C)$ to $[\mathbf{Space}_{\nabla}(C)]^{op}$. The same also holds for $\mathbf{Alg}^H_{\nabla}(D, C)$ and $[\mathbf{Space}_{\nabla}^H(C)]^{op}$.
\end{lemma}
\begin{proof}
First, as $\mathsf{A}$ is a bounded distributive lattice, by Priestley duality, Theorem \ref{PriestleyDuality}, the pair $(\mathcal{F}_p(\mathsf{A}), \subseteq)$ is a Priestley space. To prove that $\mathfrak{S}(\mathcal{A})$ is a $\nabla$-space, we need to check the conditions in Definition \ref{DefNablaSpace}.
The first condition is trivial, as $R$ is compatible with $\subseteq$. For the rest, we use the lattice isomorphism $\alpha: \mathsf{A} \to CU(\mathcal{F}_p(\mathsf{A}), \subseteq)$, defined by $\alpha(a)=\{P \in \mathcal{F}_p(\mathsf{A}) \mid a \in P\}$ as introduced in Theorem \ref{PriestleyDuality}. Notice that the set $\alpha(a)$ is a clopen upset, for any $a \in \mathsf{A}$. Now, for the second condition, we have to show that $R[P]$ is closed, for any $P \in \mathcal{F}_p(\mathsf{A})$. Note that $R[P]=\bigcap_{x \in P} \alpha(\nabla x)$, as $(P, Q) \in R$ iff $\nabla[P] \subseteq Q$ iff $\nabla x \in Q$, for any $x \in P$. Since each $\alpha(\nabla x)$ is closed in the Priestley topology, the set $R[P]$ must be also closed. For the third condition, if $U$ is a clopen subset, we have to show that $\lozenge_R(U)$ is also clopen. As $U$ is open, there exist some open sets in the basis of Priestley topology in the form $\alpha(a_i) \cap \alpha(b_i)^c$ such that $U=\bigcup_{i \in I} \alpha(a_i) \cap \alpha(b_i)^c $. As $U$ is also closed and the Priestley space is compact and Hausdorff, $U$ is compact and hence we can assume that $I$ is finite. Then, since $\lozenge_R$ commutes with unions, we have $\lozenge_R (U)= \bigcup_{i \in I} \lozenge_R( \alpha(a_i) \cap \alpha(b_i)^c)$. By Theorem \ref{KripkeEmbedding}, $i: \mathcal{A} \to \mathfrak{U}\mathfrak{P}(\mathcal{A})$ preserves the implication. As $\alpha$ is just $i$ restricted on its codomain from all upsets of $\mathfrak{P}(\mathcal{A})$ to all clopen upsets, we have $\alpha(a_i \to b_i)=\alpha(a_i) \to \alpha(b_i)$. Hence, $[\alpha(a_{i} \to b_i)]^c=\lozenge_R(\alpha(a_i) \cap \alpha(b_i)^c)$. As a result, $\lozenge_R(\alpha(a_i) \cap \alpha(b_i)^c)$ is clopen. Therefore, as $I$ is finite, $\lozenge_R(U)$ is a finite union of clopen subsets which is clopen itself. For the fourth condition, assume that $V$ is a clopen upset. Then, by Priestley duality, Theorem \ref{PriestleyDuality}, there must be $a \in A$ such that $U=\alpha(a)$. By Theorem \ref{KripkeEmbedding}, $i: \mathcal{A} \to \mathfrak{U}\mathfrak{P}(\mathcal{A})$ preserves $\nabla$. Again, as $\alpha$ is just $i$ restricted on its codomain, $\alpha$ also respects $\nabla$, i.e., $\nabla_R(V)=\nabla_R (\alpha(a))=\alpha(\nabla a)$. Hence, $\nabla_R(V)$ is also a clopen upset. This completes the proof that $\mathfrak{S}(\mathcal{A})$ is a $\nabla$-space.

For the conditions in the set $\{N, R, L, Fa, Fu\}$, note that the conditions just refer to the Kripke structure and as a Kripke frame, $\mathfrak{P}(\mathcal{A})$ and $\mathfrak{S}(\mathcal{A})$ coincide. Hence, if $\mathcal{A} \in \mathbf{Alg}_{\nabla}(D, C)$, then by Theorem \ref{KripkeEmbedding}, we have $\mathfrak{S}(\mathcal{A}) \in \mathbf{Space}_{\nabla}(C)$. For the Heyting case, if $\mathcal{A}$ is Heyting, then $(\mathcal{F}_p(\mathsf{A}), \subseteq)$ is an Esakia space, by Theorem \ref{PriestleyDuality}.
For morphisms, we have to prove that $\mathfrak{S}(f)$ is a continuous Kripke morphism which is a result of Theorem \ref{KripkeEmbedding} and Theorem \ref{PriestleyDuality}. Moreover, note that if both $\mathcal{A}$ and $\mathcal{B}$ are Heyting and $f: \mathcal{A} \to \mathcal{B}$ preserves the Heyting implication, then $f^{-1}$ is an Esakia map, by Theorem \ref{PriestleyDuality}.
\end{proof}

\begin{theorem}\label{OurPriestleyDuality} (Priestley-Esakia duality for distributive $\nabla$-algebras)
Let $C \subseteq \{N, H, R, L, Fa, Fu\}$. Then, the functors $\mathfrak{S}$ and $\mathfrak{A}$ and the following natural isomorphisms:
\begin{description}
\item[]
$\alpha : \mathcal{A} \to \mathfrak{A}\mathfrak{S}(\mathcal{A})$ defined by $\alpha(a)=\{P \in \mathcal{F}_p(\mathsf{A}) \mid a \in P\}$,
\item[]
$\beta: (X, \leq, R) \to \mathfrak{S}\mathfrak{A}(X, \leq, R)$ defined by $\beta(x)=\{U \in CU(X, \leq) \mid x \in U\} $.
\end{description}
establish an equivalence between the categories $\mathbf{Alg}_{\nabla}(D, C)$ and $\mathbf{Space}_{\nabla}^{op}(C)$. The same also holds for $\mathbf{Alg}_{\nabla}^{H}(D, C)$ and $[\mathbf{Space}^{H}_{\nabla}(C)]^{op}$.
\end{theorem}
\begin{proof}
If we consider the map $\alpha$ (resp. $\beta$) as a map between the underlying bounded distributive lattices of $\mathcal{A}$ and $\mathfrak{A}\mathfrak{S}(\mathcal{A})$ (resp. Priestley spaces of $(X, \leq, R)$ and $\mathfrak{S}\mathfrak{A}(X, \leq, R)$), it is an isomorphism in the category $\mathbf{DLat}$ (resp. $\mathbf{Pries}$), by Priestley duality, Theorem \ref{PriestleyDuality}. Therefore, we only need to prove that $\alpha$ and its inverse preserve $\nabla$ and $\to$, on the one hand, and $\beta$ and its inverse satisfy the three conditions in the definition of a Kripke morphism, Definition \ref{DefKripke}, on the other. For the former, note that $\alpha$ is just the $\nabla$-algebra map $i_{\mathcal{A}}$ presented in Theorem \ref{KripkeEmbedding} whose codomain is restricted to a subalgebra and hence it preserves $\nabla$ and $\to$. Since $\alpha$ is a bijective algebraic morphism, its inverse is automatically a $\nabla$-algebra morphism. Moreover, if $\mathcal{A}$ is Heyting, with a similar argument we can show that $\alpha$ also preserves the Heyting implication.

For the latter, let $\mathscr{X}=(X, \leq, R)$ be a $\nabla$-space. Recall that the $\nabla$ on $\mathfrak{A}(X, \leq, R)$, denoted by $\nabla_R$, is defined by $\nabla_R(U)=\{x \in X \mid \exists y \in U, (y, x) \in R\}$. Moreover, the $R$ on $\mathfrak{S}\mathfrak{A}(X, \leq, R)$, denoted by $R_{\mathfrak{A}(\mathscr{X})}$, is defined by $(P, Q) \in R_{\mathfrak{A}(\mathscr{X})}$ iff $\nabla_R[P] \subseteq Q$, for any $P, Q \in \mathcal{F}_p(\mathfrak{A}(\mathscr{X}))$. Now, we show that:\\

\noindent \textbf{Claim.} $(x, y) \in R$ iff $(\beta(x), \beta(y)) \in R_{\mathfrak{A}(\mathscr{X})}$, for any $x, y \in X$.\\

\noindent To prove, first note that $(\beta(x), \beta(y)) \in R_{\mathfrak{A}(\mathscr{X})}$ iff
$\nabla_{R}[\beta (x)] \subseteq \beta(y) 
$ iff (for any clopen upset $U$ of $\mathscr{X}$, if $x \in U$ then $y \in \nabla_{R}(U)$). Call the statement in the paranthesis $(*)$. Now, it is enough to prove that $(x, y) \in R$ is equivalent to $(*)$. For the first direction, if $(x, y) \in R$, then if $x \in U$, by the definition of $\nabla_{R}$, we have $y \in \nabla_{R}(U)$. For the converse, consider the closed upset $R[x]$. Since $(X, \leq)$ is a Priestley space, by Lemma \ref{PropPriestly}, the set $R[x]$ is an intersection of a family of clopen upsets $V_i$, i.e., $R[x]=\bigcap_{i \in I} V_i$. For any $V_i$, consider $W_i=\Box_R V_i=\{z \in X \mid R[z] \subseteq V_i\}$. Since $V_i$ is a clopen upset and the clopen upsets is closed under $\Box_R$ by Lemma \ref{DefA}, $W_i$ is also a clopen upset of $X$. Let $i \in I$ be an arbitrary index. 
As $R[x] \subseteq V_i$, we have $x \in W_i$. Therefore, by $(*)$, we have $y \in \nabla_R(W_i)$ which means that there exists $z \in W_i$ such that $(z, y) \in R$. Since $z \in W_i$, by the definition of $W_i$, we have $R[z] \subseteq V_i$. Therefore, as $(z, y) \in R$, we have $y \in V_i$. As $i \in I$ is an arbitrary index, $y \in \bigcap_{i \in I} V_i$. Finally, as $R[x]=\bigcap_{i \in I} V_i$, we reach $y \in R[x]$ which implies $(x, y) \in R$. This completes the proof of the claim.

Now, using the claim, it is easy to see that all the three conditions in Definition of a Kripke morphism, Definition \ref{DefKripke}, are satisfied for $\beta$ and its inverse. We only prove the claim for $\beta$. The converse is similar. For the first condition, using the claim, if $(x, y) \in R$ then $(\beta(x), \beta(y)) \in R_{\mathfrak{A}(\mathscr{X})}$. For the second condition, let $(\beta(x), Q) \in R_{\mathfrak{A}(\mathscr{X})}$, for $Q \in \mathcal{F}_p(\mathfrak{A}(\mathscr{X}))$. Then, by the surjectivity of $\beta$, there is $y \in X$ such that $\beta(y)=Q$. Hence, $(\beta(x), \beta(y)) \in R_{\mathfrak{A}(\mathscr{X})}$. By the claim, we have $(x, y) \in R$. As $\beta(y)=Q$, this completes the second condition. The proof for the third condition is similar. Moreover, one can use a similar argument to show that $\beta$ is a Heyting Kripke morphism, if $(X, \leq, R)$ is Heyting.
\end{proof}

For the future reference, let us apply Theorem \ref{OurPriestleyDuality} to normal (full or faithful) $\nabla$-algebras to get a representation theorem for these $\nabla$-algebras in terms of generalized Esakia spaces.

\begin{corollary}\label{PriestleyRepForNormal}
For any normal distributive $\nabla$-algebra $\mathcal{A}$, there is a generalized Esakia space $\mathscr{X}=(X, \leq, \pi)$ such that $\mathcal{A} \cong (CU(X, \leq_X), \pi^{-1}, \to_{\mathscr{X}})$, where $U \to_{\mathscr{X}} V=\{x \in X \mid \pi^{-1}(\uparrow \! x) \cap U \subseteq V\}$. If $\mathcal{A}$ is faithful (resp. full), $\mathscr{X}$ can be chosen as regular monic (resp. epic).
\end{corollary}
\begin{proof}
By Theorem \ref{OurPriestleyDuality}, for any normal distributive $\nabla$-algebra $\mathcal{A}$, there is a normal $\nabla$-space $(X, \leq, R)$ such that $\mathcal{A} \cong \mathfrak{A}(X, \leq, R)$. Set $\mathscr{X}=(X, \leq, \pi)$ as the corresponding generalized Esakia space. Note that $\pi$ is the normality witness of $(X, \leq, R)$. Therefore, $(x, y) \in R$ iff $x \leq \pi(y)$, for any $x, y \in X$. One can easily use this equivalence together with the upward-closedness of $U$ to see that $\nabla_R(U)=\pi^{-1}(U)$ and  $U \to_{R} V=\{x \in X \mid \pi^{-1}(\uparrow \! x) \cap U \subseteq V\}$. Therefore, $\mathfrak{A}(X, \leq, R) \cong (CU(X, \leq_X), \pi^{-1}, \to_{\mathscr{X}})$. For faithfulness and fullness conditions, note that if $\mathcal{A}$ is faithful (resp. full), then $(X, \leq, R)$ can be chosen as faithful (resp. full), by Lemma \ref{DefS}. Then, by Lemma \ref{PriestleyFullFaithful}, the map $\pi: X \to X$ is regular monic in $\mathbf{Pries}$ (resp. epic in $\mathbf{Pries}$) which implies that $\mathscr{X}$ is regular monic (resp. epic).
\end{proof}

In the following, we provide an algebraic characterization for different categories of generalized Esakia spaces. Most interestingly, $\mathbf{Alg}^H_{\nabla}(D, N, Fa, Fu)$ provides an algebraic characterization for the category of reversible dynamic Esakia spaces.

\begin{corollary}\label{CorollaryPriestleyEsakia}
We have the following equivalences:
\begin{description}
\item[$\bullet$]
$\mathbf{Alg}_{\nabla}(D, N) \simeq \mathbf{gEsakia}^{op}$ and $\mathbf{Alg}^H_{\nabla}(D, N) \simeq (\mathbf{gEsakia}^H)^{op}$.
\item[$\bullet$]
$\mathbf{Alg}_{\nabla}(D, N, Fa) \simeq \mathbf{gEsakia}_{rm}^{op}$ and $\mathbf{Alg}^H_{\nabla}(D, N, Fa) \simeq (\mathbf{gEsakia}^H_{rm})^{op}$.
\item[$\bullet$]
$\mathbf{Alg}_{\nabla}(D, N, Fu) \simeq \mathbf{gEsakia}_{e}^{op}$ and $\mathbf{Alg}^H_{\nabla}(D, N, Fu) \simeq (\mathbf{gEsakia}^H_{e})^{op}$.
\item[$\bullet$]
$\mathbf{Alg}^H_{\nabla}(D, N, Fa, Fu) \simeq (\mathbf{Esakia}^{\mathbb{Z}})^{op}$.
\end{description}
\end{corollary}
\begin{proof}
The corollary is a direct consequence of Theorem \ref{OurPriestleyDuality} and Corollaries \ref{CorollaryforNormalNablaI}, \ref{CorollaryforNormalNablaII} and \ref{CorollaryforNormalNablaIII}.
\end{proof}

\section{Spectral Duality for Distributive $\nabla$-algebras} \label{SectionSpectral}

In this section, we rephrase the Priestley-Esakia duality, Theorem \ref{OurPriestleyDuality}, in the context of spectral spaces to derive a spectral duality theorem for distributive $\nabla$-algebras. We begin by recalling the definition of a spectral space and the well-known isomorphism between the categories of Priestley spaces and spectral spaces.

\subsection{Spectral spaces}

A topological space $X$ is called \emph{sober} if for every closed set $Y \subseteq X$ that is not the union of two smaller closed sets, there is a unique point $x \in X$ such that $Y=Cl(\{x\})$, where by $Cl(Z)$, we mean the closure of $Z$, for any $Z \subseteq X$. A space is called \textit{spectral}, if it is sober and compact and its compact open subsets form a basis. Note that the last condition implies that any finite intersections of compact opens is a compact open itself. A continuous map $f: X \to Y$ between spectral spaces is called \emph{spectral}, if $f^{-1}(U)$ is compact, for any compact open $U$. A subset $Y \subseteq X$ is called spectral, if it is a spectral space with the subspace topology and the inclusion map $i: Y \to X$ is a spectral map. Following \cite{Bezh}, we call a subset $Y \subseteq X$ \textit{doubly spectral}, if both $Y$ and $Y^c$ are spectral subsets of $X$. We denote the category of spectral spaces and spectral maps by $\mathbf{Spec}$. 

It is a well-known fact that the category of Priestley spaces is isomorphic to the category of spectral spaces. To see how,
define the assignment $F : \mathbf{Pries} \to \mathbf{Spec}$ on the objects by $F(X, \leq)=X^s$, where $X^s$ is the set $X$ equipped with the \emph{spectral topology}, i.e., the topology consisting of all the open upsets of $(X, \leq)$. For the morphism $f:(X, \leq_X) \to (Y, \leq_Y)$, define $F(f)=f$. Conversely,
define the assignment $G : \mathbf{Spec} \to  \mathbf{Pries}$ on the objects by $G(X)=(X^p, \leq)$, where $X^p$ is the set $X$ equipped with the \emph{patch topology}, i.e., the topology generated by the basis elements of the form $U \cap V^c$, where $U$ and $V$ are open subsets of $X$ and $\leq$ is the \emph{specialization order} of the original topology of $X$, i.e., $x \leq y$ iff $x \in Cl(\{y\})$, where $Cl$ is the closure operator for the original topology of $X$. For the morphism $f: X \to Y$, define $G(f)=f$.

\begin{theorem} (\cite{Dickmann, Bezh}) \label{SpectralDuality}
The assignments $F$ and $G$ are functors and establish an isomorphism between the categories $\mathbf{Pries}$ and $\mathbf{Spec}$. Therefore,  $\mathbf{Pries} \cong \mathbf{Spec}$.
\end{theorem}

The previous isomorphism provides a useful dictionary between Priestley spaces and spectral spaces. We will recall this dictionary to transfer what we proved in Section \ref{PriEsaDuality} to the spectral setting. 

\begin{definition}
Let $X$ be a topological space. A subset $Y \subseteq X$ is called \emph{saturated}, if it is an intersection of some open subsets and it is called \emph{co-saturated} if it is a union of some closed subsets. Moreover, for any subset $Y \subseteq X$, define $Sat_X(Y)=\bigcap \{ U \in \mathcal{O}(X) \mid U \supseteq Y \}$ and denote $Sat_X(\{x\})$ by $Sat_X(x)$,  for any $x \in X$. It is easy to see that $\uparrow \! Y=Sat_X(Y)$, where the order is the specialization order \cite{ConLat}.
\end{definition}

\begin{theorem} (The Priestley-Spectral Dictionary \cite{Bezh}) \label{Dictionary}
Let $(X, \leq)$ be a Priestley space and $X^s=F(X, \leq)$ be its corresponding spectral space.  Then, a subset of $X$ is:
\begin{itemize}
\item[$\bullet$]
an upset in the Priestley space $(X, \leq)$ iff it is saturated in $X^s$,
\item[$\bullet$]
a downset in the Priestley space $(X, \leq)$ iff it is co-saturated in $X^s$,
\item[$\bullet$]
an open subset in the Priestley space $(X, \leq)$ iff its complement is a spectral subset of $X^s$,
\item[$\bullet$]
a closed subset in the Priestley space $(X, \leq)$ iff it is a spectral subset of $X^s$,
\item[$\bullet$]
a clopen in the Priestley space $(X, \leq)$ iff it is a doubly spectral subset of $X^s$,
\item[$\bullet$]
an open upset in the Priestley space $(X, \leq)$ iff it is open in $X^s$,
\item[$\bullet$]
a closed upset in the Priestley space $(X, \leq)$ iff it is compact and saturated in $X^s$,
\item[$\bullet$]
a clopen upset of the Priestley space $(X, \leq)$ iff it is compact open in $X^s$.
\end{itemize}
Moreover, for any $Y \subseteq X$, we have $Cl(Y)=\downarrow \! Cl_p(Y)$, where $Cl$ and $Cl_p$ are the closure operators for the spectral topology and the Priestley topology, respectively.
\end{theorem}

\begin{definition}(\cite{Bezh})
Let $X$ be a spectral space. It is called \emph{Heyting}, if $Cl(Y)$ is a doubly spectral subset of $X$, for any doubly spectral subset $Y$ of $X$. A spectral map $f: X \to Y$ between two Heyting spectral spaces is called Heyting, if $f[Sat_X(x)]=Sat_Y(f(x))$, for any $x \in X$. We denote the category of Heyting spectral spaces and Heyting spectral maps by $\mathbf{HSpec}$. 
\end{definition}

Using the dictionary in Lemma \ref{Dictionary}, it is easy to see that Heyting spectral spaces and Esakia spaces are equivalent under $F-G$ isomorphism:

\begin{theorem} (\cite{Bezh}) \label{HSpectralDuality}
The assignments $F$ and $G$ establish an isomorphism between the categories $\mathbf{Esakia}$ and $\mathbf{HSpec}$. Therefore,  $\mathbf{Esakia} \cong \mathbf{HSpec}$.
\end{theorem}

Combining the isomorphisms $\mathbf{Pries} \cong \mathbf{Spec}$ and $\mathbf{Esakia} \cong \mathbf{HSpec}$ with Priestley/Esakia duality theorem, we have:

\begin{theorem}(Spectral duality \cite{Bezh}) \label{RealSpectralDuality}
$\mathbf{DLat} \simeq \mathbf{Spec}^{op}$ and $\mathbf{Heyting} \simeq \mathbf{HSpec}^{op}$.
\end{theorem}

Finally, similar to what we did for Priestley spaces, let us provide a concrete characterization for the epic and regular monic maps in the category $\mathbf{Spec}$.

\begin{lemma}\label{SurInjSpectral}
Let $X$ and $Y$ be two spectral spaces and $f: X \to Y$ be a spectral map. 
\begin{description}
\item[$(i)$]
$f$ is surjective iff $f$ is an epic map in $\mathbf{Spec}$.
\item[$(ii)$]
$f$ is a topological embedding iff $f$ is a regular monic in $\mathbf{Spec}$.
\end{description}
\end{lemma}
\begin{proof}
$(i)$ is an easy consequence of the isomorphism between $\mathbf{Pries}$ and $\mathbf{Spec}$ in Theorem \ref{SpectralDuality} together with Lemma \ref{PriestleyInjSurj}. For $(ii)$,
as spectral spaces are $T_0$, by Theorem \ref{InjSurjforContinuous}, 
$f: X \to Y$ is a topological embedding iff
$f^{-1}: \mathcal{O}(Y) \to \mathcal{O}(X)$ is surjective. 
Let $(X^p, \leq_X)$ and $(Y^p, \leq_Y)$ be the corresponding Priestly spaces of the spectral spaces $X$ and $Y$, respectively. Therefore, $f:X \to Y$ is a reqular monic in $\mathbf{Spec}$ iff the map $f: (X^p, \leq_X) \to (Y^p, \leq_Y)$ is a regular monic in $\mathbf{Pries}$, by Theorem \ref{SpectralDuality}.
Therefore, it is enough to show that $f: (X^p, \leq_X) \to (Y^p, \leq_Y)$ is a regular monic in $\mathbf{Pries}$ iff $f^{-1}: \mathcal{O}(Y) \to \mathcal{O}(X)$ is surjective. For the first direction, assume that the map $f: (X^p, \leq_X) \to (Y^p, \leq_Y)$ is a regular monic in $\mathbf{Pries}$. By Lemma \ref{PriestleyInjSurj}, the function $f^{-1}: CU(Y^p, \leq_Y) \to CU(X^p, \leq_X)$ is surjective. As any open upset of $(X^p,\leq_X)$ is a union of clopen upsets by Lemma \ref{PropPriestly}, we can conclude that the function $f^{-1}: \mathcal{OU}(Y^p, \leq_Y) \to \mathcal{OU}(X^p, \leq_X)$ is surjective, where $\mathcal{OU}(-)$ means the set of all open upsets of the Priestley space.
Now, observe that by Lemma \ref{Dictionary}, the open subsets of a spectral space are actually the open upsets of the corresponding Priestley space. Therefore, $f^{-1}: \mathcal{O}(Y) \to \mathcal{O}(X)$ is surjective. 

For the converse, assume $f^{-1}: \mathcal{O}(Y) \to \mathcal{O}(X)$ is surjective.
We first show that $f: (X^p, \leq_X) \to (Y^p, \leq_Y)$ is an order-embedding. Notice that $\leq_X$ and $\leq_Y$ are the specialization orders of the spectral spaces $X$ and $Y$, respectively. Assume $f(x) \leq_Y f(y)$. By definition of the specialization order, for any open $U \in Y$ of the spectral space $Y$, if $f(x) \in U$ then $f(y) \in U$, or equivalently, $x \in f^{-1}(U)$ implies $y \in f^{-1}(U)$. Since $f^{-1}$ is surjective, we can conclude that $x \in V$ implies $y \in V$, for any open $V \subseteq X$. Hence, $x \leq_X y$. Therefore,  $f: (X^p, \leq_X) \to (Y^p, \leq_Y)$ is an order-embedding. Finally, by Lemma \ref{PriestleyInjSurj}, $f$ is a regular monic in $\mathbf{Pries}$.
\end{proof}

\subsection{$\nabla$-spectral spaces and spectral duality}
In the previous subsection, we recalled the well-known fact that (Heyting) spectral and (Esakia) Priestley spaces are just the two different presentations of one mathematical entity. Here, we use this fact to present the spectral presentation of a $\nabla$-space introduced in Section \ref{PriEsaDuality}.
\begin{definition}\label{DefNablaSpectralspace}
A \emph{$\nabla$-spectral space} is a pair $(X, R)$ of a spectral space $X$ and a binary relation $R$ on $X$ such that:
\begin{itemize}
\item[$\bullet$]
$R[x]=\{y \in X \mid (x, y) \in R\}$ is compact and saturated, for every $x \in X$,
\item[$\bullet$]
$R^{-1}[y]=\{x \in X \mid (x, y) \in R\}$ is co-saturated, for every $y \in X$,
\item[$\bullet$]
$\lozenge_R(Y)=\{x \in X \mid \exists y \in Y, (x, y) \in R\}$ is a doubly spectral subset of $X$, for any doubly spectral subset $Y \subseteq X$,
\item[$\bullet$]
$\nabla_R(V)=\{x \in X \mid \exists y \in V, (y, x) \in R \}$ is a compact open subset, for any compact open subset $V$.
\end{itemize}
A spectral $\nabla$-space $(X, R)$ is called:
\begin{description}
\item[$(N)$]
\emph{normal} if there exists a continuous map $\pi: X \to X$, called the normality witness, such that $(x, y) \in R$ iff $x \in Cl(\{\pi(y)\})$, for any $x, y \in X$,
\item[$(H)$]
\emph{Heyting} if $Cl(Y)$ is a doubly spectral subset of $X$, for any doubly spectral $Y \subseteq X$, 
\item[$(R)$]
\emph{right} if $R$ is reflexive,
\item[$(L)$]
\emph{left} if $(x, y) \in R$ implies $x \in Cl(\{y\})$, for any $x, y \in X$, 
\item[$(Fa)$]
\emph{faithful} if for any $x \in X$, there exists $y \in X$ such that $(y, x) \in R$ and for any $z \in X$ such that $(y, z) \in R$, we have $x \in Cl(\{z\})$,
\item[$(Fu)$]
\emph{full} if for any $x \in X$, there exists $y \in X$ such that $(x, y) \in R$ and for any $z \in X$ such that $(z, y) \in R$, we have $z \in Cl(\{x\})$.
\end{description}
If $(X, R_X)$ and $(Y, R_Y)$ are two $\nabla$-spectral spaces, by a \emph{$\nabla$-spectral map} $f: (X, R_X) \to (Y, R_Y)$, we mean a spectral map $f: X \to Y$ such that:
\begin{itemize}
\item[$\bullet$]
for any $x, z \in X$, if $(x, z) \in R_X$ then $(f(x), f(z)) \in R_Y$,
\item[$\bullet$]
for any $y \in Y$ such that $(f(x), y) \in R_Y $, there exists $z \in X$ such that $(x, z) \in R_X $ and $f(z)=y$,
\item[$\bullet$]
for any $y \in Y$ such that $(y, f(x)) \in R_Y $, there exists $z \in X$ such that $(z, x) \in R_X $ and $y \in Cl_Y(\{f(z)\})$.
\end{itemize}
If $(X, R_X)$ and $(Y, R_Y)$ are two Heyting $\nabla$-spectral spaces, then a $\nabla$-spectral map $f: (X, R_X) \to (Y, R_Y)$ is called Heyting, if it is Heyting as a spectral map.
For any $C \subseteq \{N, H, R, L, Fa, Fu\}$, the class of all $\nabla$-spectral spaces satisfying the conditions in $C$ together with $\nabla$-spectral maps form a category denoted by $\mathbf{Spec}_{\nabla}(C)$. If we restrict the objects to Heyting $\nabla$-spectral spaces and the morphisms to Heyting $\nabla$-spectral maps, we denote the subcategory by $\mathbf{Spec}^{H}_{\nabla}(C)$.
\end{definition}

Define the assignment $\tilde{F}: \mathbf{Spec}_{\nabla} \to \mathbf{Space}_{\nabla}$ on the objects by $\tilde{F}(X, R)=(F(X), R)=(X^p, \leq, R)$ and on the morphism $f: (X, R_X) \to (Y, R_Y)$ by $\tilde{F}(f)=F(f)=f$. Conversely, define $\tilde{G}: \mathbf{Space}_{\nabla} \to \mathbf{Spec}_{\nabla}$ on the objects by $\tilde{G}(X, \leq, R)=(G(X, \leq), R)=(X^s, R)$ and on the morphism $f: (X, \leq_X, R_X) \to (Y, \leq_Y, R_Y)$ by $\tilde{G}(f)=G(f)=f$.

\begin{lemma}(Extended Dictionary)\label{DicExtended}
The assignments $\tilde{F}$ and $\tilde{G}$ are well-defined functors and establish an isomorphism between the categories $\mathbf{Space}_{\nabla}$ and $\mathbf{Spec}_{\nabla}$.
Moreover, for any $C \subseteq \{N, H, R, L, Fa, Fu\}$,  the isomorphism restricts to an isomorphism between $\mathbf{Spec}_{\nabla}(C)$ and $\mathbf{Space}_{\nabla}(C)$ and similarly between $\mathbf{Spec}^H_{\nabla}(C)$ and $\mathbf{Space}^H_{\nabla}(C)$. Therefore, $\mathbf{Spec}_{\nabla}(C)\cong \mathbf{Space}_{\nabla}(C)$ and $\mathbf{Spec}^H_{\nabla}(C) \cong \mathbf{Space}^H_{\nabla}(C)$.
\end{lemma}
\begin{proof}
To show that $\tilde{F}$ and $\tilde{G}$ are well-defined, we have to show two things. First, 
$(X, R) \in \mathbf{Spec}_{\nabla}(C)$ iff $(X^p, \leq, R) \in \mathbf{Space}_{\nabla}(C)$, for any spectral space $X$, any binary relation $R \subseteq X \times X$ and any $C \subseteq \{N, H, R, L, Fa, Fu\}$. Second, we have to show that a map $f:(X, R_X) \to (Y, R_Y)$ is a $\nabla$-spectral space map iff $f:(X^p, \leq_X, R_X) \to (Y^p, \leq_Y, R_Y)$ is a $\nabla$-space map. To show these equivalences, we use the dictionary in Lemma \ref{Dictionary}.
For the first, note that $R[x]$ is compact and saturated and $R^{-1}[y]$ is co-saturated as subsets of the spectral space $X$ iff $R[x]$ is a closed upset and $R^{-1}[y]$ is a downset as the subsets of $(X^p, \leq)$. Therefore, the combination of the first two conditions in Definition \ref{DefNablaSpectralspace} of a $\nabla$-spectral space is equivalent to $R$ being compatible with the order and $R[x]$ being closed with respect to $(X^p, \leq)$. By Lemma \ref{Dictionary}, the other two conditions in Definition \ref{DefNablaSpectralspace} are simply the spectral presentation of the other two conditions in Definition \ref{DefNablaSpace}.
For the conditions in the set $\{R, L, Fa, Fu\}$, as they are just the direct translations of the same conditions in Definition \ref{DefNablaSpace}, there is nothing to prove. For $(N)$, let $(X^p, \leq, R)$ be a normal $\nabla$-space with the normality witness $\pi$. Therefore, $(x, y) \in R$ iff $x \leq \pi(y)$, for any $x, y \in X$. By Lemma \ref{NormalityForPriestly}, the map $\pi$ is a Priestley map on $(X^p, \leq)$. Hence, by Theorem \ref{SpectralDuality}, $\pi: X \to X$ is a spectral and hence continuous map with respect to the spectral topology on $X$. As $x \in Cl(\{\pi(y)\})$ iff $x \leq \pi(y)$, for any $x, y \in X$, we can conclude that $(X, R)$ is a normal $\nabla$-spectral space. Conversely, if $(X, R)$ is a normal $\nabla$-spectral space, then its normality witness $\pi$ is a continuous map over $X$ with respect to its spectral topology and hence it preserves its specialization order. Therefore, it is an order-preserving map on $(X^p, \leq)$ which implies the normality of the $\nabla$-space $(X^p, \leq, R)$. For $(H)$, Theorem \ref{HSpectralDuality} shows that a spectral space is Heyting iff $(X^p, \leq)$ is Esakia. For the morphisms, as the conditions in Definition \ref{DefNablaSpectralspace} are just the immediate translation of the conditions in Definition \ref{DefNablaSpace}, there is nothing to prove. For Heyting morphisms, see Theorem \ref{HSpectralDuality}, again.

It is easy to see that $\tilde{F}$ and $\tilde{G}$ are functors. The fact that these functor establish an isomorphism between $\mathbf{Spec}_{\nabla}(C)$ and $\mathbf{Space}_{\nabla}(C)$ and between $\mathbf{Spec}^H_{\nabla}(C)$ and $\mathbf{Space}^H_{\nabla}(C)$ is a direct consequence of Theorem \ref{SpectralDuality}.
\end{proof}

\begin{lemma}\label{ExplicitAlgebraToSpec}
Let $C \subseteq \{N, H, R, L, Fa, Fu\}$. 
\begin{itemize} 
\item[$\bullet$]
Define the assignment $\mathfrak{S}': \mathbf{Alg}_{\nabla}(D, C) \to \mathbf{Spec}_{\nabla}^{op}(C)$
\begin{itemize}
\item
on the objects by $\mathfrak{S}'(\mathcal{A})=(\mathcal{F}_p(\mathsf{A}), R)$, where $\mathcal{F}_p(\mathsf{A})$ equipped with the spectral topology as defined by the basis $\{P \in \mathcal{F}_p(\mathsf{A}) \mid a \in P \}$, for any $a \in A$ and $(P, Q) \in R$ iff $\nabla[P] \subseteq Q$. 
\item
on the morphism $f: \mathcal{A} \to \mathcal{B}$ by $\mathfrak{S}'(f)=f^{-1}: \mathfrak{S}'(\mathcal{B}) \to \mathfrak{S}'(\mathcal{A})$.
\end{itemize}
\item[$\bullet$]
Define the assignment $\mathfrak{A}': \mathbf{Spec}_{\nabla}^{op}(C) \to \mathbf{Alg}_{\nabla}(D, C)$
\begin{itemize}
\item
on the objects by
$
\mathfrak{A}'(X, R)=(CO(X), \subseteq, \nabla_{R}, \to_{R})
$,
where $CO(X)$ is the set of all compact opens of $X$, $\nabla_{R}(U)=\{x \in X \mid \exists y \in U \; (y, x) \in R \}$ and $U \to_{R} V=\{x \in X \mid  R[x] \cap U \subseteq V \}$ over $CO(X)$.
\item
on the morphism $f: \mathscr{X} \to \mathscr{Y}$ by
$\mathfrak{A}'(f)=f^{-1}: \mathfrak{A}'(\mathscr{Y}) \to \mathfrak{A}'(\mathscr{X})$,
\end{itemize}
\end{itemize}
Then, the assignments $\mathfrak{A}'$ and $\mathfrak{S}'$ are the functors $\mathfrak{A}$ and $\mathfrak{S}$ modified by the bridge provided in Lemma \ref{DicExtended}, i.e., $\mathfrak{A}'= \mathfrak{A}\circ \tilde{F}$ and $\mathfrak{S}'=\tilde{G} \circ\mathfrak{S} $.
\end{lemma}
\begin{proof}
The proof is a simple application of the dictionary provided in Lemma \ref{Dictionary} and the Priestley-Esakia duality, Theorem \ref{PriestleyDuality}. For $\mathfrak{A}'$, the only point to explain is that the compact opens of a spectral space $X$ are exactly the clopen upsets of the corresponding Priestley space $(X^p, \leq)$. For $\mathfrak{S}'$, we show that if $Y=(\mathcal{F}_p(\mathsf{A}), \subseteq)$ is the Priestley space associated to the bounded distributive algebra $\mathsf{A}$, then the topology of $Y^s$ is equal to the topology generated by the basis $\{P \in \mathcal{F}_p(\mathsf{A}) \mid a \in P \}$. Denote the latter topology by $\tau$. First, note that any subset in the form $\{P \in \mathcal{F}_p(\mathsf{A}) \mid a \in P \}$ is an open upset in the Priestley space $Y$ and hence it is open in $Y^s$, by definition. Secondly, if $U$ is an open in $Y^s$, it is an open upset of $Y$. Then, $U$ is a union of clopen upsets of $Y$, by Lemma \ref{PropPriestly}. Hence, it is enough to show that any clopen upset of $Y$ is in $\tau$. But, this is clear by Theorem \ref{PriEsaDuality}, as the clopen upsets of the Priestley space $Y=(\mathcal{F}_p(\mathsf{A}), \subseteq)$ are in the form $\alpha(a)=\{P \in \mathcal{F}_p(\mathsf{A}) \mid a \in P \}$, for some $a \in A$.
\end{proof}

\begin{theorem}\label{NablaSpectralDuality}(Spectral duality for distributive $\nabla$-algebras) Assume $C \subseteq \{N, H, R, L, Fa, Fu\}$. Then, the functors $\mathfrak{S}'$ and $\mathfrak{A}'$ and the following natural isomorphisms:
\begin{itemize}
\item[]
$\alpha' : \mathcal{A} \to \mathfrak{A}'\mathfrak{S}'(\mathcal{A})$ defined by $\alpha'(a)=\{P \in \mathcal{F}_p(\mathsf{A}) \mid a \in P\}$
\item[]
$\beta': (X, R) \to \mathfrak{S}'\mathfrak{A}'(X, R)$ defined by $\beta'(x)=\{U \in CO(X) \mid x \in U\} $
\end{itemize}
establish an equivalence between the categories $\mathbf{Alg}_{\nabla}(D, C)$ and $\mathbf{Spec}_{\nabla}^{op}(C)$. The same also holds for $\mathbf{Alg}^H_{\nabla}(D, C)$ and $ [\mathbf{Spec}^H_{\nabla}(C)]^{op}$.
\end{theorem}
\begin{proof}
Using Lemma \ref{DicExtended}, Lemma \ref{ExplicitAlgebraToSpec} and the Priestley-Esakia duality, Theorem \ref{OurPriestleyDuality}, the claim easily follows.
\end{proof}

\subsection{Generalized H-spectral spaces}

In Subsection \ref{SubsubsectionofGEsakia}, we observed that normal $\nabla$-spaces are actually the dynamic Priestley spaces satisfying an Esakia-style condition. A similar claim also holds for $\nabla$-spectral spaces. In fact, they can be described as dynamic spectral spaces satisfying a Heyting-style condition.

\begin{definition}\label{gSpectral}
A \emph{generalized H-spectral space} is a tuple $(X, \pi)$, where $X$ is a spectral space and $\pi: X \to X$ is a spectral map such that $Cl(\pi[Y])$ is doubly spectral subset of $X$, for any doubly spectral $Y \subseteq X$. A generalized H-spectral space is called \emph{regular monic} (resp. \emph{epic}), if $\pi$ is a regular monic (resp. epic) in $\mathbf{Spec}$ or equivalently a topological embedding (resp. surjective), by Lemma \ref{SurInjSpectral}. A generalized H-spectral space is called \emph{Heyting} if $X$ is a Heyting spectral space. By a \emph{generalized H-spectral map} $f: (X, \pi_X) \to (Y, \pi_Y)$, we mean a spectral map $f: X \to Y$ such that $f\pi_X=\pi_Y f$ and $\pi_Y^{-1}(Sat_Y(f(x)))=f[\pi_X^{-1}(Sat_X(x))]$, for any $x \in X$. It is called Heyting, if it is also a Heyting map. Generalized H-spectral spaces and generalized H-spectral maps form a category, denoted by $\mathbf{gHSpec}$. For any $I \subseteq \{rm, e\}$, the full subcategory of generalized H-spectral spaces satisfying the conditions in $I$ is denoted by $\mathbf{gHSpec}_{I}$. If we restrict the objects and the morphisms to Heyting objects and Heyting maps, the subcategories are denoted by $\mathbf{gHSpec}^H$ and $\mathbf{gHSpec}^H_{I}$, respectively. 
\end{definition}

The following lemma shows that generalized H-spectral spaces and generalized Esakia spaces are two presentations of the same mathematical entity.

\begin{lemma}\label{ExtDictionaryII}
Let $X$ be a spectral space and $\pi: X \to X$ be a function. Then, $(X, \pi)$ is a generalized H-spectral space iff $(X^p, \leq, \pi)$ is a generalized Esakia space. The same also holds if we add the adjectives regular monic or epic on both sides.
\end{lemma}
\begin{proof}
By Theorem \ref{SpectralDuality}, it is clear that $\pi: X \to X$ is a spectral map iff $\pi: (X^p, \leq) \to (X^p, \leq)$ is a Priestley map. Moreover, by Lemma \ref{Dictionary},
the doubly spectral subspaces of $X$ are exactly the clopens of $X^p$ and $Cl(Z)=\downarrow Cl_p(Z)$, for any $Z \subseteq X$. Hence, it is enough to show that $\downarrow \! \pi[Y]$ is clopen in $X^p$ iff $Cl(\pi[Y])=\downarrow Cl_p(\pi[Y])$ is clopen in $X^p$, for any clopen subset $Y \subseteq X^p$. 
If we show that $\pi[Y]$ is closed in $X^p$, for any clopen $Y \subseteq X^p$, then the last equivalence is clear.
For that, as $Y$ is closed in $X^p$, it is also compact which implies that $\pi[Y]$ is also compact and hence closed as the subspace of $X^p$.  
For regular monic (resp. epic), by Theorem \ref{SpectralDuality},  $\pi: X \to X$ is a regular monic (resp. epic) in $\mathbf{Spec}$ iff $\pi: (X^p, \leq) \to (X^p, \leq)$ is a regular monic (resp. epic) in $\mathbf{Pries}$.
\end{proof}

Lemma \ref{ExtDictionaryII} provides a representation theorem for normal distributive $\nabla$-algebras.

\begin{corollary}\label{SpectralRepForNormal}
For any normal distributive $\nabla$-algebra $\mathcal{A}$, there is a generalized H-spectral space $\mathscr{X}=(X, \pi)$ such that $\mathcal{A} \cong (CO(X), \pi^{-1}, \to_{\mathscr{X}})$, where $U \to_{\mathscr{X}} V=\{x \in X \mid \pi^{-1}(Sat_X(x)) \cap U \subseteq V\}$. If $\mathcal{A}$ is faithful (resp. full), $\mathscr{X}$ can be chosen as regular monic (resp. epic).
\end{corollary}
\begin{proof}
By Corollary \ref{PriestleyRepForNormal}, for any $\nabla$-algebra $\mathcal{A}$, there is a generalized Esakia space $\mathscr{Y}=(Y, \leq, \pi)$ such that $\mathcal{A} \cong (CU(Y, \leq), \pi^{-1}, \to_{\mathscr{Y}})$, where $U \to_{\mathscr{Y}} V=\{x \in X \mid \pi^{-1}(\uparrow \! x) \cap U \subseteq V\}$. Note that if $\mathcal{A}$ is faithful (resp. full), then $\mathscr{Y}$ is regular monic (resp. epic). Set $X=Y^s$. By Lemma \ref{ExtDictionaryII}, $\mathscr{X}=(X, \pi)$ is a generalized H-spectral space and $\mathscr{X}$ is regular monic (resp. epic) iff $\mathscr{Y}$ is regular monic (resp. epic). Using the dictionary between Priestley spaces and spectral spaces, one can see that $CO(X)=CU(Y, \leq)$ and $\uparrow \! x=Sat_X(x)$. Therefore, $\mathcal{A} \cong (CO(X), \pi^{-1}, \to_{\mathscr{X}})$.
\end{proof}

\begin{lemma}\label{ExtDictionaryIII}
Let $X$ and $Y$ be spectral spaces, $\pi_X: X \to X$ and $\pi_Y: Y \to Y$ be spectral maps and $f: X \to Y$ be a function. Then, $f: (X, \pi_X) \to (Y, \pi_Y)$ is a (Heyting) generalized H-spectral map iff $f: (X^p, \leq_X, \pi_X) \to (Y^p, \leq_Y, \pi_Y)$ is a (Heyting) generalized Esakia map. 
\end{lemma}
\begin{proof}
By Theorem \ref{SpectralDuality}, $f: X \to Y$ is spectral iff $f: (X^p, \leq_X) \to (Y^p, \leq_Y)$ is a Priestley map. For the other conditions on the two sides, it is enough to consider the fact that $Sat_X(x)=\uparrow \! x$ and $Sat_Y(y)=\uparrow \! y$, for any $x \in X$ and $y \in Y$, where the left hand side is computed with respect to the spectral topology while the right hand side uses the specialization order.
\end{proof}

\begin{corollary}\label{ExtDicIV}
$\mathbf{gHSpec}_I \cong \mathbf{gEsakia}_I$ and $\mathbf{gHSpec}^H_I \cong \mathbf{gEsakia}^H_I$, for any $I \subseteq \{rm, e\}$.
\end{corollary}

\begin{corollary}\label{CorollarygHSpec}
\begin{description}
\item[$\bullet$]
$\mathbf{Spec}_{\nabla}(N) \cong \mathbf{gHSpec}$ and $\mathbf{Spec}^H_{\nabla}(N) \cong \mathbf{gHSpec}^H$.
\item[$\bullet$]
$\mathbf{Spec}_{\nabla}(N, Fa) \cong \mathbf{gHSpec}_{rm}$ and $\mathbf{Spec}^H_{\nabla}(N, Fa) \cong \mathbf{gHSpec}^H_{rm}$.
\item[$\bullet$]
$\mathbf{Spec}_{\nabla}(N, Fu) \cong \mathbf{gHSpec}_{e}$ and $\mathbf{Spec}^H_{\nabla}(N, Fu) \cong \mathbf{gHSpec}^H_{e}$.
\item[$\bullet$]
$\mathbf{Spec}_{\nabla}(N, Fa, Fu) \cong \mathbf{gHSpec}_{rm, e}$ and $\mathbf{Spec}^H_{\nabla}(N, Fa, Fu) \cong \mathbf{gHSpec}^H_{rm, e}$.
\item[$\bullet$]
$\mathbf{Spec}^H_{\nabla}(N, Fa, Fu) \cong \mathbf{HSpec}^{\mathbb{Z}}$.
\end{description}
\end{corollary}
\begin{proof}
For the first four parts, use Corollary \ref{ExtDicIV}, Lemma \ref{DicExtended} and Corollaries \ref{CorollaryforNormalNablaI}, and \ref{CorollaryforNormalNablaII}. For the last part, use Corollary \ref{CorollaryforNormalNablaIII}, Lemma \ref{DicExtended}, and Theorem \ref{HSpectralDuality}.
\end{proof}

\begin{corollary}\label{CorollarySpectralDuality}
We have the following equivalences:
\begin{description}
\item[$\bullet$]
$\mathbf{Alg}_{\nabla}(D, N) \simeq \mathbf{gHSpec}^{op}$ and $\mathbf{Alg}^H_{\nabla}(D, N) \simeq (\mathbf{gHSpec}^H)^{op}$.
\item[$\bullet$]
$\mathbf{Alg}_{\nabla}(D, N, Fa) \simeq \mathbf{gHSpec}_{rm}^{op}$ and $\mathbf{Alg}^H_{\nabla}(D, N, Fa) \simeq (\mathbf{gHSpec}^H_{rm})^{op}$.
\item[$\bullet$]
$\mathbf{Alg}_{\nabla}(D, N, Fu) \simeq \mathbf{gHSpec}_{e}^{op}$ and $\mathbf{Alg}^H_{\nabla}(D, N, Fu) \simeq (\mathbf{gHSpec}^H_{e})^{op}$.
\item[$\bullet$]
$\mathbf{Alg}^H_{\nabla}(D, N, Fa, Fu) \simeq (\mathbf{HSpec}^{\mathbb{Z}})^{op}$.
\end{description}
\end{corollary}
\begin{proof}
Use Corollary \ref{CorollarygHSpec} and Theorem \ref{NablaSpectralDuality}. 
\end{proof}

Notice that Corollary \ref{CorollarySpectralDuality} provides an algebraic presentation for reversible dynamic Heyting spectral spaces.

\section{Ring-theoretic Representations}\label{RingTheoretic}

Spectral spaces can be characterized as the spectrum of commutative unital rings. In this section, we use the spectral representation of normal distributive $\nabla$-algebras provided in Section \ref{SectionSpectral} to represent these $\nabla$-algebras as some sort of dynamic rings. For that purpose, we first recall some basic facts from commutative algebra, including the connection between commutative unital rings and spectral spaces.

\subsection{Some backgrounds from commutative algebra}

Let $R$ be a commutative unital ring. By an \emph{ideal} of $R$, we mean a subset $I$ of $R$ such that $x - y \in I$ and $rx \in I$, for any $x, y \in I$ and any $r \in R$. For any set $A \subseteq R$, by $\langle A \rangle$, we mean the least ideal containing $A$. An ideal $I \subset R$ is called \emph{prime} if $xy \in I$ implies either $x \in I$ or $y \in I$. The \emph{radical of an ideal} $I$ is defined by $Rad(I)=\{x \in R \mid \exists n \in \mathbb{N}, x^n \in I\}$. An ideal is called \emph{radical} iff $Rad(I)=I$.  The ideal $Rad(I)$ is equal to the intersection of all prime ideals containing $I$. For any two ideals $I$ and $J$, define the ideal $[J:I]$ as $\{x \in R \mid \forall y \in I, xy \in J\}$. Note that if $J$ is a radical ideal, then so is $[J: I]$. By $Spec(R)$, we mean the topological space of all the prime ideals of $R$ equipped with the topology $\{ U_r \mid r \in R \}$, where $U_r=\{P \in Spec (R) \mid r \notin P\}$. The following facts are all well-known in commutative algebra \cite{Dickmann}.
\begin{lemma}\label{RingFacts}
\begin{itemize}
\item[$\bullet$]
The space $Spec(R)$ is spectral and for any ring homomorphism $f : R \to S$, the induced map $Spec (f)=f^{-1}: Spec (S) \to Spec (R)$ is a spectral map.
\item[$\bullet$]
(\textit{Hochster's Theorem}) Conversely, any spectral space is homeomorphic to $Spec(R)$ for a commutative unital ring $R$. Even stronger, for any spectral spaces $X$ and $Y$ such that $X \neq Y$ and any spectral map $F: X \to Y$, there exist commutative unital rings $R_X$ and $R_Y$, a ring homomorphism $f: R_Y \to R_X$ and homeomorphisms $\alpha_X: Spec (R_X) \cong X$ and $\alpha_Y: Spec (R_Y) \cong Y$ such that $F \circ \alpha_X=\alpha_Y \circ Spec (f)$:
\[\small\begin{tikzcd}[ampersand replacement=\&]
	{Spec(R_X)} \&\& {Specx(R_Y)} \\
	\\
	X \&\& Y
	\arrow["F"', from=3-1, to=3-3]
	\arrow["{\alpha_X}"', from=1-1, to=3-1]
	\arrow["{\alpha_Y}", from=1-3, to=3-3]
	\arrow["{Spec(f)}", from=1-1, to=1-3]
\end{tikzcd}\]
\item[$\bullet$]
The poset of all radical ideals of $R$, denoted by $\mathcal{RI}(R)$, is a locale and its Heyting implication is $I \supset J=[J:I]$, i.e., $K \cap I \subseteq J$ iff $K \subseteq [J:I]$,
for any $I, J, K \in \mathcal{RI}(R)$.
\item[$\bullet$]
Define $\mathcal{I}: \mathcal{O}(Spec(R)) \to \mathcal{RI}(R)$ by $\mathcal{I}(U)=\{r \in R \mid \forall P \notin U, r \in P\}$ and $\mathcal{U}:  \mathcal{RI}(R) \to \mathcal{O}(Spec(R))$ by $\mathcal{U}(I)=\{P \in Spec (R) \mid \exists x \in I,  x \notin P \}$. Then $\mathcal{I}$ and $\mathcal{U}$ establish a localic isomorphism between  $\mathcal{O}(Spec(R))$ and $\mathcal{RI}(R)$.
\item[$\bullet$]
For any ring homomorphism $f: R \to S$, the map $f_* : \mathcal{RI}(R) \to \mathcal{RI}(S)$ defined by $f_*(I)=Rad(\langle f[I] \rangle)$ is a localic map and $f_* \dashv f^{-1}$, where $f^{-1}: \mathcal{RI}(S) \to \mathcal{RI}(R)$. Moreover, $f_*$ is the $\mathcal{I}-\mathcal{U}$ counterpart of $Spec (f)^{-1}: \mathcal{O}(Spec (R)) \to \mathcal{O}(Spec (S))$, i.e., $f_*= \mathcal{I} \circ Spec(f)^{-1} \circ \mathcal{U}$.
\item[$\bullet$]
If $F: Spec (R) \to Spec(S)$ is a continuous map, the map $F^{-1}: \mathcal{O}(Spec(S)) \to \mathcal{O}(Spec(R))$ is clearly a localic map. The $\mathcal{I}-\mathcal{U}$ correspondence assigns the localic map $\hat{F}: \mathcal{RI}(S) \to \mathcal{RI}(R)$ to $F^{-1}$, defined by $\hat{F}=\mathcal{I} \circ F^{-1} \circ \mathcal{U}$. This assignment is functorial in a contravariant way, i.e., $\widehat{id_{Spec(R)}}=id_{\mathcal{RI}(R)}$ and $\widehat{GH}=\hat{H}\hat{G}$, for any two continuous maps $G: Spec(S) \to Spec(T)$ and $H: Spec(R) \to Spec(S)$.
\end{itemize}
\end{lemma}

Using Lemma \ref{RingFacts}, we can provide a ring-theoretic characterization for the ring homomorphisms whose spectrum are topological embedding (resp. surjective). The characterization uses an ``algebraic version" of being  surjective (resp. injective). First, let us motivate these new algebraic versions. Let $f: R \to S$ be a ring homomorphism. It is easy to see that $f$ is injective iff $f[A_1]=f[A_2]$ implies $A_1=A_2$, for any two subsets $A_1, A_2 \subseteq R$ and it is surjective iff for any subset $B \subseteq S$, there is a subset $A \subseteq R$ such that $f[A]=B$. To make these two definitions more algebraic-friendly, we read both sides of the equalities up to the radical of the generated ideals, i.e., the injectivity becomes the condition that if $Rad(\langle f[A_1] \rangle)=Rad(\langle f[A_2] \rangle)$, then $Rad(\langle A_1 \rangle)=Rad(\langle A_2 \rangle)$, for any two subsets $A_1, A_2 \subseteq R$. Similarly, the surjectivity becomes the condition that for any subset $B \subseteq S$, there is a subset $A \subseteq R$ such that $Rad(\langle f[A] \rangle)=Rad(\langle B \rangle)$. It is easy to see that these algebraic-friendly versions of injectivity and surjectivity are nothing but the injectivity and surjectivity of $f_*: \mathcal{RI}(R) \to \mathcal{RI}(S)$, respectively. 

In the following lemma, we provide a more concrete characterization for these two conditions. We also show that they are the algebraic presentation of the homomorphisms whose spectrum are topological embedding or surjective. 
First of all, recall the following well-known lemma about the existence of prime ideals:

\begin{lemma}\label{TheExistenceOfPrimeIdeal}(\cite{Dickmann})
Let $R$ be a commutative unital ring, $S \subseteq R$ be a set closed under multiplication and $I \subseteq R$ be an ideal such that $I \cap S=\varnothing$. Then, there exists a prime ideal $P \supseteq I$ such that $P \cap S=\varnothing$.
\end{lemma}

\begin{lemma}\label{InjSurjInRings}
Let $f: R \to S$ be a ring homomorphism. Then:
\begin{description}
\item[$(i)$]
$Spec(f):Spec(S) \to Spec(R)$ is surjective iff $f^{-1}: \mathcal{RI}(S) \to \mathcal{RI}(R)$ is surjective iff $f_*: \mathcal{RI}(R) \to \mathcal{RI}(S)$ is injective iff $f(r)\in \langle f[A] \rangle$ implies the existence of $n \geq 0$ such that $r^n \in \langle A \rangle$, for any $A \cup \{r\} \subseteq R$.
\item[$(ii)$]
$Spec(f):Spec(S) \to Spec(R)$ is a topological embedding iff $f_*: \mathcal{RI}(R) \to \mathcal{RI}(S)$ is surjective iff for any $s \in S$, there exist elements $r_1, \ldots, r_k \in R$ and a natural number $m \geq 0$ such that $s \mid f(r_{i})$, for any $1 \leq i \leq k$ and $s^m=\sum_{i=1}^k s_if(r_{i})$, for some $s_i \in S$.
\end{description}
\end{lemma}
\begin{proof}
For $(i)$, assume $Spec(f):Spec(S) \to Spec(R)$ is surjective. We want to show that $f^{-1}: \mathcal{RI}(S) \to \mathcal{RI}(R)$ is surjective.
For that, we claim $f^{-1}(f_*(J))=J$, for any $J \in \mathcal{RI}(R)$. As $f[J] \subseteq Rad(\langle f[J] \rangle)=f_*(J)$, it is clear that $J \subseteq f^{-1}(f_*(J))$. For $f^{-1}(f_*(J)) \subseteq J$, as $J=\bigcap_{J \subseteq Q} Q$, where the intersection is on the prime ideals, it is enough to prove $f^{-1}(f_*(J)) \subseteq Q$, for any prime ideal $Q \supseteq J$. Let $Q$ be such a prime ideal. As $Spec(f)$ is surjective, there exists a prime ideal $P$ of $S$ such that $Spec(f)(P)=f^{-1}(P)=Q$. As $J \subseteq Q$, we have $f[J] \subseteq P$ which implies $f_*(J) \subseteq P$, as $P$ is a radical ideal. Hence, $f^{-1}(f_*(J)) \subseteq f^{-1}(P)=Q$.

To prove the second equivalence, as $f^{-1} \dashv f_*$, the surjectivity of $f^{-1}$ is equivalent to the injectivity of $f_*$. Hence, we can use them interchangeably. For the third implication, suppose $f_*: \mathcal{RI}(R) \to \mathcal{RI}(S)$ is injective. Hence, $f^{-1}: \mathcal{RI}(S) \to \mathcal{RI}(R)$ is surjective. Also, assume $f(r)\in \langle f[A] \rangle$, where $A \cup \{r\} \subseteq R$. By surjectivity, there is $I \in \mathcal{RI}(S)$ such that $f^{-1}(I)=Rad(\langle A \rangle)$. Hence, it is enough to prove $r \in f^{-1}(I)$. As $A \subseteq Rad(\langle A \rangle)=f^{-1}(I)$, we have $f[A] \subseteq I$. As $I$ is an ideal, we have $\langle f[A] \rangle \subseteq I$ which implies $f(r) \in I$ and hence $r \in f^{-1}(I)$.

Finally, for the last implication, suppose that $f(r)\in \langle f[A] \rangle$ implies the existence of a natural number $n \geq 0$ such that $r^n \in \langle A \rangle$, for any $A \cup \{r\} \subseteq R$. Let $Q$ be a prime ideal of $R$. We want to provide a prime ideal $P$ of $S$ such that $Spec(f)(P)=f^{-1}(P)=Q$. Set $I=\langle f[Q] \rangle$ and $B=f[Q^c]$. As $Q$ is prime, $Q^c$ is closed under multiplication and so is $B=f[Q^c]$. First, we claim that $I \cap B=\varnothing$. Let $x \in I \cap B$. As $x \in B$, there is $y \notin Q$ such that $x=f(y)$. Moreover, $x \in I$ means that $x \in \langle f[Q] \rangle$. Hence, $f(y) \in \langle f[Q] \rangle$. By the assumption, the latter implies the existence of $n \geq 0$ such that $y^n \in \langle Q \rangle$. As $Q$ is a prime ideal, $y \in Q$ which is a contradiction. Therefore, $I \cap B=\varnothing$. Hence, by Lemma \ref{TheExistenceOfPrimeIdeal}, there is a prime ideal $P$ of $S$ such that $I \subseteq P$ and $P \cap B=\varnothing$. Using the former, we have $f[Q] \subseteq P$ which implies $Q \subseteq f^{-1}(P)$ and by the latter, we conclude $f^{-1}(P) \subseteq Q$, because, the existence of $z \in f^{-1}(P) \cap Q^c$ implies $f(z) \in f[Q^c] \cap P=B \cap P$ which is impossible. Therefore, $f^{-1}(P)=Q$. Hence, $Spec(f)=f^{-1}:Spec(S) \to Spec(R)$ is surjective. 

For $(ii)$, as $Spec(S)$ is spectral and hence $T_0$, by Lemma \ref{InjSurjforContinuous}, the map $Spec(f)$ is a topological embedding iff $Spec(f)^{-1}: \mathcal{O}(Spec(R)) \to \mathcal{O}(Spec(S))$ is surjective. By Lemma \ref{RingFacts}, $f_*= \mathcal{I} \circ Spec(f)^{-1} \circ \mathcal{U}$. As $\mathcal{I}$ and $\mathcal{U}$ are localic isomorphisms, the surjectivity of $Spec(f)^{-1}$ is equivalent to the surjectivity of $f_*: \mathcal{RI}(R) \to \mathcal{RI}(S)$. Hence, we proved the equivalence between the first two parts of $(ii)$. Now, assume $f_*$ is surjective. Then, for any $s \in S$, there exists a radical ideal $I$ of $R$ such that $f_*(I)=Rad(\langle s \rangle)$. Hence, $s \in f_*(I)=Rad(\langle f[I] \rangle)$ which implies the existence of $u_1, \ldots, u_k \in I$ and $p \geq 0$ such that $s^{p}=\sum_i v_if(u_i)$, for some $v_i \in S$. As $u_i \in I$, we have $f(u_i) \in f[I] \subseteq Rad(\langle s \rangle)$ which implies the existence of $n_i \geq 0$ such that $s \mid f(u_{i}^{n_i})$. Then, $s^{p+\sum_i n_i}=\sum_i s_if(u_{i}^{n_i})$, for some $s_i \in S$. Setting $m=p+\sum_i n_i$, and $r_i=u_{i}^{n_i}$ completes the proof.

Conversely, we prove the surjectivity of $f_*$ from the last part of $(ii)$. To prove, it is enough to show $f_*(f^{-1}(I))=I$, for any $I \in \mathcal{RI}(S)$. By Lemma \ref{RingFacts}, we have $f_* \dashv f^{-1}$. Hence, the direction $f_*(f^{-1}(I)) \subseteq I$ is clear. For the other direction, i.e., $I \subseteq f_*(f^{-1}(I))$, assume $s \in I$. By the assumption, there are elements $r_1, \ldots, r_k \in R$ and a natural number $m \geq 0$ such that $s \mid f(r_{i})$ and $s^m=\sum_i s_if(r_{i})$, for some $s_i \in S$. As $s \mid f(r_{i})$, we have $f(r_i) \in I$ which implies $r_i \in f^{-1}(I)$ and hence $f(r_i) \in f[f^{-1}(I)]$. As $s^m=\sum_i s_if(r_{i})$, we have $s^m \in \langle f[f^{-1}(I)] \rangle$ and hence $s \in f_*(f^{-1}(I))$. Therefore, $I=f_*(f^{-1}(I))$ which completes the proof of surjectivity of $f_*$.
\end{proof}

\subsection{A representation theorem}

In the last subsection, we recalled the ring-theoretic representation of spectral spaces. An immediate consequence of this representation is the following representation theorem for Heyting algebras:
\begin{theorem} \label{RingRepresentationForHeyting}(Ring-theoretic representation of Heyting algebras)
For any Heyting algebra $\mathcal{H}$, there exist a commutative unital ring $R$ and an embedding $i: \mathcal{H} \to \mathcal{RI}(R)$ of Heyting algebras.
\end{theorem}
\begin{proof}
Let $\mathcal{H}$ be a Heyting algebra. Combining Theorem \ref{HSpectralDuality} and Theorem \ref{PriestleyDuality}, there is a spectral space $X$ such that $\mathcal{H} \cong CO(X)$, as Heyting algebras, where $CO(X)$ is the poset of all compact open subsets of the spectral space $X$. By Hochster's Theorem, Lemma \ref{RingFacts}, there exists a commutative unital ring $R$ such that $X$ is homeomorphic to $Spec(R)$. Therefore, by Lemma \ref{RingFacts}, $\mathcal{RI}(R) \cong \mathcal{O}(Spec(R)) \cong \mathcal{O}(X)$, as locales and hence as Heyting algebras. 
Finally, as the inclusion $i: CO(X) \to \mathcal{O}(X)$ is a Heyting algebra embedding, using the isomorphisms mentioned above, there is a Heyting algebra embedding from $\mathcal{H}$ into $\mathcal{RI}(R)$.
\end{proof}

In the remainder of this subsection, we aim to generalize Theorem \ref{RingRepresentationForHeyting} to encompass all normal distributive $\nabla$-algebras. To motivate such a representation theorem, let us start with an observation. Let $R$ be a commutative unital ring and $\pi: R \to R$ be a ring homomorphism. Then, it is easy to see that the tuple $(\mathcal{RI}(R), \pi_*, \to_{\pi})$ is a normal distributive $\nabla$-algebra, where $I \to_{\pi} J=\pi^{-1}([J:I])$. To show that it is a $\nabla$-algebra, we only use the fact that $[J:I]$ is the Heyting implication of $\mathcal{RI}(R)$ and $\pi_* \dashv \pi^{-1}$. For the normality, by Lemma \ref{RingFacts}, we have $\pi_*=\mathcal{I}Spec(\pi)^{-1}\mathcal{U}$. As the right hand-side preserves all finite meets, the normality is also proved. 

As one may expect, the goal is to show that any normal distributive $\nabla$-algebra can be embedded into a $\nabla$-algebra assigned to a pair $(R, \pi)$ in the above-mentioned way. Unfortunately, this seems to be too ambitious to prove. Instead, we prove a version modified by an arbitrary homeomorphism. To explain, recall that to any continuous map $f : Spec(R) \to Spec(S)$, we assigned a localic map $\hat{f}:\mathcal{RI}(S) \to \mathcal{RI}(R)$ in a functorial way. The functoriality implies that if $f$ is a homeomorphism, the map $\hat{f}$ is a localic isomorphism.

\begin{theorem}\label{SpectralIsAlgebra}
Let $R$ and $S$ be commutative unital rings, $\pi: R \to S$ be a ring homomorphism and $f : Spec(R) \cong Spec(S)$ be a homeomorphism. Then, the tuple $(\mathcal{RI}(R), \nabla, \to)$ is a normal $\nabla$-algebra, where
\[
\nabla I=\hat{f}(\pi_*(I)) \;\;\; \text{and} \;\;\; I \to J= \pi^{-1}(\hat{g} ([J: I]))
\]
and $g$ is the inverse of $f$. Moreover, the normal $\nabla$-algebra $(\mathcal{RI}(R),  \nabla, \to)$ is 
\begin{itemize}
\item[$\bullet$]
faithful iff for any $s \in S$, there exist elements $r_1, \ldots, r_k \in R$ and a natural number $m \geq 0$ such that $s \mid \pi(r_{i})$, for any $1 \leq i \leq k$ and $s^m=\sum_{i=1}^k s_i \pi(r_{i})$, for some $s_i \in S$.
\item[$\bullet$]
full iff $\pi(r)\in \langle \pi[A] \rangle$ implies the existence of $n \geq 0$ such that $r^n \in \langle A \rangle$, for any $A \cup \{r\} \subseteq R$.
\end{itemize}
\end{theorem}
\begin{proof}
First, we show that $\mathcal{RI}(R)$ is closed under the operations $\nabla$ and $\to$. For $\nabla$, there is nothing to prove. For $\to$, if $I, J \in \mathcal{RI}(R)$, then $[J : I] \in \mathcal{RI}(R)$ and hence $\hat{g}([J:I]) \in \mathcal{RI}(S)$. Finally, since $\pi^{-1}$ maps the radical ideals of $S$ to the radical ideals of $R$, we have $I \to J \in \mathcal{RI}(R)$. For the adjunction condition, as $fg=id_{Spec(S)}$ and $gf=id_{Spec(R)}$ and the operation $\widehat{(-)}$ is functorial, we have $\hat{f} \hat{g}=id_{\mathcal{RI}(R)}$, $\hat{g}\hat{f}=id_{\mathcal{RI}(S)}$. Therefore, for any $I, J, K \in \mathcal{RI}(R)$, we have:
\[
\hat{f}(\pi_*(I)) \cap J \subseteq K \;\;\;\;\; \text{iff} \;\;\;\;\; I \subseteq  \pi^{-1}(\hat{g} [K: J])
\]
because $\pi_* \dashv \pi^{-1}$, $\hat{f} \hat{g}=id_{\mathcal{RI}(R)}$, and $\hat{g}\hat{f}=id_{\mathcal{RI}(S)}$. For normality, note that $\hat{f}\pi_*=\hat{f}\circ\mathcal{I}\circ Spec(\pi)^{-1}\circ \mathcal{U}$. As the right hand-side preserves all finite meets, there is nothing to prove.
For faithfulness (resp. fullness), as $\hat{f}$ and $\hat{g}$ are localic isomorphisms, the $\nabla$-algebra is faithful (resp. full) iff $\pi_*$ is surjective (resp. injective). Using the characterizations in Lemma \ref{InjSurjInRings}, the proof will be complete.
\end{proof}

\begin{definition}
Let $R$ and $S$ be commutative unital rings, $\pi: R \to S$ be a ring homomorphism and $f : Spec(R) \simeq Spec(S)$ be a homeomorphism. Then, the tuple $\mathcal{R}=(R, S, \pi, f)$ is called a \emph{semi-dynamic ring}. A semi-dynamic ring is called \emph{faithful} if for any $s \in S$, there exist elements $r_1, \cdots, r_k \in R$ and a natural number $m \geq 0$ such that $s \mid f(r_{i})$, for any $1 \leq i \leq k$ and $s^m=\sum_{i=1}^k s_if(r_{i})$, for some $s_i \in S$. It is called \emph{full}
if $f(r)\in \langle f[A] \rangle$ implies the existence of $n \geq 0$ such that $r^n \in \langle A \rangle$, for any $A \cup \{r\} \subseteq R$. For any $C \subseteq \{Fa, Fu\}$, we denoted the class of all semi-dynamic rings satisfying the conditions in $C$ by $\mathfrak{R}(C)$.
\end{definition}

By Theorem \ref{SpectralIsAlgebra}, to any semi-dynamic ring $\mathcal{R}$, one can associate 
the normal $\nabla$-algebra $\mathfrak{Al}(\mathcal{R})=(\mathcal{RI}(R),  \nabla, \to)$ and it is faithful or full if $\mathcal{R}$ is faithful or full, respectively.

\begin{theorem}\label{RingRepresentation} (Ring-theoretic Representation) 
Let $C \subseteq \{Fa, Fu\}$. Then, for any $\mathcal{A} \in \mathcal{V}(N, D, C)$, there exists a semi-dynamic ring $\mathcal{R} \in \mathfrak{R}(C)$ and a $\nabla$-algebra embedding $i: \mathcal{A} \to \mathfrak{Al}(\mathcal{R})$.
\end{theorem}
\begin{proof}
Suppose $\mathcal{A}$ is a normal $\nabla$-algebra. By Corollary \ref{SpectralRepForNormal}, there is a generalized H-spectral space $\mathscr{X}=(X, h)$ such that $\mathcal{A} \cong (CO(X), h^{-1}, \to_{\mathscr{X}})$, where $U \to_{\mathscr{X}} V=\{x \in X \mid \pi^{-1}(Sat_X(x)) \cap U \subseteq V\}$. Let $Y$ be a space homeomorphic to $X$ but $X \neq Y$. Call the homeomorphism $g: Y \to X$. This is just a technical condition to make the Hochster's Theorem applicable. Since $h: X \to X$ is spectral and $g$ is a homeomorphism, $hg: Y \to X$ is also spectral. By Hochster's Theorem, Lemma \ref{RingFacts}, there exist rings $R$ and $S$, a ring homomorphism $\pi: R \to S$ and homeomorphisms $\alpha_X: Spec (R) \cong X$ and $\alpha_Y: Spec (S) \cong Y$ such that $(hg) \circ \alpha_Y=\alpha_X \circ Spec (\pi)$:
\[\small \begin{tikzcd}
	{Spec(S)} && Y \\
	\\
	{Spec(R)} && X
	\arrow["hg", from=1-3, to=3-3]
	\arrow["{Spec(\pi)}"', from=1-1, to=3-1]
	\arrow["{\alpha_X}"', from=3-1, to=3-3]
	\arrow["{\alpha_Y}", from=1-1, to=1-3]
\end{tikzcd}\]
Define $f: Spec (R) \to Spec (S)$ as $f=\alpha_Y^{-1} \circ g^{-1} \circ \alpha_X$. This map is clearly a homeomorphism. Therefore, we can consider the semi-dynamic ring $\mathcal{R}=(R, S, \pi, f)$. We claim that $\mathcal{R}$ is the semi-dynamic ring we are looking for. Before proving this claim, 
notice that if $\mathcal{A}$ is faithful (resp. full), then by Corollary \ref{SpectralRepForNormal}, the map $h: X \to X$ must be a topological embedding (resp. surjection). Since $g$ is a homeomorphism, $hg$ is also a topological embedding (resp. surjection) which implies that $Spec(\pi)$ is a topological embedding (resp. surjection). By Lemma \ref{InjSurjInRings}, the latter implies the faithfulness (resp. fullness) of $\mathfrak{Al}(\mathcal{R})$.

To provide a $\nabla$-algebra embedding from $\mathcal{A}$ into $\mathfrak{Al}(\mathcal{R})$, we pursue the following three steps: 

I. As $(X, h)$ is a dynamic topological system, $\mathcal{B}=(\mathcal{O}(X), h^{-1}, \to_h)$ is a normal $\nabla$-algebra, by Example \ref{DTS}. We claim that the inclusion map from $(CO(X), h^{-1}, \to_{\mathscr{X}})$ into $\mathcal{B}$ is a $\nabla$-algebra-embedding. It is clearly a bounded lattice embedding preserving the $\nabla$. We have to show that it also preserves the implication, i.e., $U \to_{\mathscr{X}} V=U \to_{h} V$, for any $U, V \in CO(X)$. Thinking inside $(CO(X), h^{-1}, \to_{\mathscr{X}})$, by the adjunction, we have $h^{-1}(U \to_{\mathscr{X}} V) \cap U \subseteq V$. Then, thinking inside $\mathcal{B}$, by the adjunction, we reach $U \to_{\mathscr{X}} V \subseteq U \to_{h} V$. For the converse, as $X$ is spectral, any open is a union of compact opens. Hence, $U \to_{h} V=\bigcup_{i \in I} W_i$, where $W_i$ is a compact open. Therefore, for any arbitrary $i\in I$, we have $W_i \subseteq U \to_{h} V$ which implies $h^{-1}(W_i) \cap U \subseteq V$. As $W_i \in CO(X)$, we reach $W_i \subseteq U \to_{\mathscr{X}} V$. As $i \in I$ is arbitrary, $U \to_{h} V=\bigcup_{i \in I} W_i \subseteq U \to_{\mathscr{X}} V$. Hence, $U \to_{\mathscr{X}} V=U \to_{h} V$.

II. Consider the map $f^{-1} \circ Spec(\pi)^{-1}: \mathcal{O}(Spec(R)) \to \mathcal{O}(Spec(R))$. As it preserves all unions, it has a right adjoint $\Box': \mathcal{O}(Spec(R)) \to \mathcal{O}(Spec(R))$. Define $U \to' V=\Box'(U \supset V)$, where $\supset$ is the Heyting implication on the locale $\mathcal{O}(Spec(R))$. As $f^{-1} \circ Spec(\pi)^{-1} \dashv \Box'$, the tuple $(\mathcal{O}(Spec(R)), f^{-1} \circ Spec(\pi)^{-1}, \to')$ is a $\nabla$-algebra. Consider the following claim:\\

\noindent \textbf{Claim.} $\mathfrak{Al}(\mathcal{R}) \cong (\mathcal{O}(Spec(R)), f^{-1} \circ Spec(\pi)^{-1}, \to')$.\\

\noindent  To prove it, consider the $\mathcal{I}-\mathcal{U}$ correspondence in Lemma \ref{RingFacts}. The map $\mathcal{I}: \mathcal{O}(Spec(R)) \to \mathcal{RI}(R)$ is a locale isomorphism and hence an isomorphism between the bounded distributive algebras underlying the $\nabla$-algebras. Recall the $\nabla$ on $\mathcal{RI}(R)$ is $\hat{f}(\pi_*(I)): \mathcal{RI}(R) \to \mathcal{RI}(R)$.
By Lemma \ref{RingFacts}, 
$
\mathcal{I} \circ f^{-1} \circ Spec(\pi)^{-1}= \hat{f} \circ \mathcal{I} \circ Spec(\pi)^{-1}=\hat{f} \circ \pi_* \circ \mathcal{I}$.
Therefore, $\mathcal{I}$ preserves the $\nabla$'s of the $\nabla$-algebras. As $\mathcal{I}$ is an order-preserving bijection, by the uniqueness of the right adjoint, $\mathcal{I}$ also preserves the implication of the $\nabla$-algebra. Hence, $\mathcal{I}$ is a $\nabla$-algebra isomorphism from $(\mathcal{O}(Spec(R)), f^{-1} \circ Spec(\pi)^{-1}, \to')$ to $\mathfrak{Al}(\mathcal{R})$. Therefore, $(\mathcal{O}(Spec(R)), f^{-1} \circ Spec(\pi)^{-1}, \to') \cong \mathfrak{Al}(\mathcal{R})$.

III. We show that $(\mathcal{O}(Spec(R)), f^{-1} \circ Spec(\pi)^{-1}, \to') \cong \mathcal{B}$. To provide the isomorphism, consider $\alpha_X^{-1}: \mathcal{O}(X) \to \mathcal{O}(Spec(R))$. As $\alpha_X: Spec(R) \to X$ is a homeomorphism, $\alpha_X^{-1}$ is a local isomorphism and hence an isomorphism between the bounded distributive algebras underlying the $\nabla$-algebras. Again, it is enough to prove that $\alpha_X^{-1}$ preserves the $\nabla$'s of the $\nabla$-algebras. For that purpose, recall that $(hg) \circ \alpha_Y=\alpha_X \circ Spec (\pi)$ and $f=\alpha_Y^{-1} \circ g^{-1} \circ \alpha_X$ and consider the following diagram:
\[\small \begin{tikzcd}
	{\mathcal{O}(X)} && {\mathcal{O}(Spec(R))} \\
	\\
	{\mathcal{O}(X)} && {\mathcal{O}(Spec(S))}
	\arrow["{h^{-1}}"', from=1-1, to=3-1]
	\arrow["{\alpha_X^{-1}}", from=1-1, to=1-3]
	\arrow["{\alpha_X^{-1}}"', from=3-1, to=3-3]
	\arrow["{f^{-1} \circ Spec(\pi)^{-1}}", from=1-3, to=3-3]
\end{tikzcd}\]
It is commutative as
\[
f^{-1} \circ Spec(\pi)^{-1}\alpha_X^{-1}=\alpha_X^{-1} g \alpha_Y Spec(\pi)^{-1}\alpha_X^{-1}=\alpha_X^{-1} g \alpha_Y \alpha_Y^{-1} g^{-1} h^{-1} = \alpha_X^{-1}h^{-1}.
\]
Therefore, $(\mathcal{O}(Spec(R)), f^{-1} \circ Spec(\pi)^{-1}, \to') \cong \mathcal{B}$ which by (II) implies $\mathfrak{Al}(\mathcal{R}) \cong \mathcal{B}$.

Finally, as $\mathcal{A} \cong (CO(X), h^{-1}, \to_{\mathscr{X}})$ and there is a $\nabla$-algebra embedding from $(CO(X), h^{-1}, \to_{\mathscr{X}})$ into  $\mathcal{B}$, we provided a $\nabla$-algebra embedding from $\mathcal{A}$ into $\mathfrak{Al}(\mathcal{R})$.
\end{proof}

\section{Logical Systems for $\nabla$-algebras}\label{Logics}

In this section, we first recall the logical systems introduced in \cite{ImSpace} to provide a syntactical presentation for some of the introduced varieties of $\nabla$-algebras. To complete this list, we also introduce two new rules to reflect the faithfulness and fullness conditions specific to the present paper. Next, we present the algebraic, ring-theoretic, topological, and Kripke semantics for these logical systems, extending the soundness and completeness results of \cite{ImSpace} to cover the new systems as well. Additionally, we utilize the Dedekind-MacNeille completion and the amalgamation property, as mentioned in Section \ref{Recall}, to establish algebraic completeness with respect to \emph{complete} $\nabla$-algebras and the \emph{deductive interpolation property}, respectively.

Let $\mathcal{L}$ be the usual language of propositional logic plus the modality $\nabla$, i.e., $\mathcal{L}=\{\wedge, \vee, \to, \top, \bot, \nabla\}$. For any multiset $\Gamma$, by $\nabla \Gamma$, we mean $\{\nabla \gamma \mid \gamma \in \Gamma\}$. By a \emph{sequent}, we mean an expression in the form $\Gamma \Rightarrow \Delta$, where $\Gamma$ and $\Delta$ are finite multisets of formulas in $\mathcal{L}$ and $|\Delta| \leq 1$. The expression $A \Leftrightarrow B$ abbreviates two sequents $A \Rightarrow B$ and $B \Rightarrow A$. For any set of sequents $\mathcal{S}$, by $V(\mathcal{S})$, we mean the set of the atoms occurring at some formulas in $\mathcal{S}$. \emph{Axioms} and \emph{rules} are defined as usual and the former is considered as a special case of the latter where there is no premises. A \emph{sequent calculus} $G$ is a finite collection of rules and
$\{\Gamma_i \Rightarrow \Delta_i\}_{i \in I} \vdash_G \Gamma \Rightarrow \Delta$ means the existence of a finite sequence $\{S_r\}_{r=0}^n$ of sequents, where $S_n=(\Gamma \Rightarrow \Delta)$ and each $S_r$ is either in $\{\Gamma_i \Rightarrow \Delta_i\}_{i \in I}$ or constructed from the sequents $\{S_l\}_{l < r}$ using an instance of a rule of $G$. When $I=\varnothing$, we usually write $G \vdash \Gamma \Rightarrow \Delta$ for $ \vdash_G \Gamma \Rightarrow \Delta$. 
Define $\mathbf{STL}$\footnote{It abbreviates space-time logic. For the motivation, see \cite{ImSpace}.} as the calculus consisting of all the rules in Figure \ref{fSTL}.
\begin{figure}[H]
\begin{center}
 \begin{tabular}{c c c}
 \AxiomC{}
  \RightLabel{\footnotesize$ Ax$}
 \UnaryInfC{$A \Rightarrow A$}
 \DisplayProof 
 &
\small \AxiomC{}
\small \RightLabel{\footnotesize$L \bot$}
\small \UnaryInfC{$ \bot \Rightarrow $}
 \DisplayProof 
&
\small  \AxiomC{}
\small \RightLabel{\footnotesize$L \top$}
\small \UnaryInfC{$ \Rightarrow \top$}
 \DisplayProof
 \\[3ex]
\end{tabular}
 
\begin{tabular}{cc}
\small \AxiomC{$\Gamma \Rightarrow \Delta$}
\small  \RightLabel{\footnotesize$ Lw$}
\small \UnaryInfC{$\Gamma, A \Rightarrow \Delta$}
 \DisplayProof 
 &
\small \AxiomC{$\Gamma \Rightarrow$}
\small \RightLabel{\footnotesize$Rw$}
\small \UnaryInfC{$\Gamma \Rightarrow A $}
 \DisplayProof 
\\[3ex]
 \end{tabular}
\begin{tabular}{cc}
\small  \AxiomC{$\Gamma, A, A \Rightarrow \Delta$}
\small \RightLabel{\footnotesize$Lc$}
\small \UnaryInfC{$\Gamma, A \Rightarrow \Delta$}
 \DisplayProof 
 &
 \small \AxiomC{$\Gamma \Rightarrow A$}
\small \AxiomC{$\Pi, A \Rightarrow \Delta$}
\small \RightLabel{\footnotesize$cut$} 
 \BinaryInfC{$\Pi, \Gamma \Rightarrow  \Delta$}
 \DisplayProof
 \\[3ex]
 \end{tabular}
 \begin{tabular}{ccc}
\small \AxiomC{$\Gamma, A \Rightarrow \Delta$}
\small \RightLabel{\footnotesize$L \wedge_1$} 
\small \UnaryInfC{$\Gamma, A \wedge B \Rightarrow \Delta$}
 \DisplayProof
&
\small \AxiomC{$\Gamma, B \Rightarrow \Delta$}
\small \RightLabel{\footnotesize$L \wedge_2$} 
\small \UnaryInfC{$\Gamma, A \wedge B \Rightarrow \Delta$}
 \DisplayProof
 &
\small \AxiomC{$\Gamma \Rightarrow A$}
\small \AxiomC{$\Gamma \Rightarrow B$}
\small \RightLabel{\footnotesize$R \wedge$} 
 \BinaryInfC{$\Gamma \Rightarrow A \wedge B$}
 \DisplayProof
  \\[3ex]
\small \AxiomC{$A \Rightarrow \Delta$}
\small \AxiomC{$B \Rightarrow \Delta$}
\small \RightLabel{\footnotesize$L \vee$} 
 \BinaryInfC{$A \vee B \Rightarrow \Delta$}
 \DisplayProof
 &
\small \AxiomC{$\Gamma \Rightarrow A$}
\small \RightLabel{\footnotesize$R \vee_1$} 
\small \UnaryInfC{$\Gamma \Rightarrow A \vee B$}
 \DisplayProof
&
\small  \AxiomC{$\Gamma \Rightarrow B$}
\small \RightLabel{\footnotesize$R \vee_2$} 
\small \UnaryInfC{$\Gamma \Rightarrow A \vee B$}
 \DisplayProof
 \\[3ex]
\end{tabular}
\begin{tabular}{cc}
 
\small \AxiomC{$A \Rightarrow B$}
\small \RightLabel{\footnotesize$\nabla$} 
\small \UnaryInfC{$\nabla A \Rightarrow \nabla B$}
 \DisplayProof
 &
\small \AxiomC{$\Gamma \Rightarrow A$}
\small \AxiomC{$\Gamma, B \Rightarrow \Delta$}
\small \RightLabel{\footnotesize$L \to$} 
\small \BinaryInfC{$\Gamma, \nabla(A \to B) \Rightarrow \Delta$}
 \DisplayProof

\small \AxiomC{$\nabla \Gamma, A \Rightarrow B$}
\small \RightLabel{\footnotesize$R \to$} 
\small \UnaryInfC{$\Gamma \Rightarrow A \to B$}
 \DisplayProof\\[3ex]
\end{tabular}
\caption{The sequent calculus $\mathbf{STL}$}
\label{fSTL}
\end{center}
\end{figure}
For any finite set $C$ of rules, 
by $\mathbf{STL}(C)$, we mean the system $\mathbf{STL}$ extended by the additional rules listed in $C$. An interesting family of rules is provided in Figure \ref{fAdd}.
\begin{figure}[H]
\begin{center}
 \begin{tabular}{c c c}
\small \AxiomC{$\Gamma \Rightarrow A$}
 \small \RightLabel{\footnotesize$ R$}
\small \UnaryInfC{$\Gamma \Rightarrow \nabla A$}
 \DisplayProof 
 &
\small \AxiomC{$\Gamma, A \Rightarrow \Delta$}
\small \RightLabel{\footnotesize$L$}
\small \UnaryInfC{$\Gamma, \nabla A \Rightarrow \Delta$}
 \DisplayProof 
 &
\small \AxiomC{$\Gamma, A \Rightarrow \Delta$}
\small \AxiomC{$\Gamma, B \Rightarrow \Delta$}
\small \RightLabel{\footnotesize$D$} 
\small \BinaryInfC{$\Gamma, A \vee B \Rightarrow \Delta$}
 \DisplayProof
 \\[3ex]
\small \AxiomC{$\Gamma \Rightarrow \Delta$}
\small  \RightLabel{\footnotesize$ N$}
\small \UnaryInfC{$\nabla \Gamma \Rightarrow \nabla \Delta$}
 \DisplayProof 
 &
\small \AxiomC{$\Gamma, A \Rightarrow B$}
\small \RightLabel{\footnotesize$Fa$} 
\small \UnaryInfC{$\Gamma \Rightarrow \nabla (A \to B)$}
 \DisplayProof
 &
\small \AxiomC{$\nabla \Gamma \Rightarrow A$}
\small \AxiomC{$\nabla \Gamma, B \Rightarrow \nabla \Delta$}
\small \RightLabel{\footnotesize$Fu$} 
\small \BinaryInfC{$\Gamma, A \to B \Rightarrow \Delta$}
 \DisplayProof\\[3ex]
\end{tabular}
\caption{The additional rules}
\label{fAdd}
\end{center}
\end{figure}
For any finite set $C$ of rules, define $\mathbf{STL}(C, H)$ as the system $\mathbf{STL}(C)$ extended by the usual rules for intuitionistic implication $\supset$, over the extended language $\mathcal{L}_i=\mathcal{L} \cup \{ \supset\}$:

\begin{flushleft}
 \textbf{Intuitionistic Implication Rules:}
\end{flushleft}
\vspace{7pt}
\begin{center}
 \begin{tabular}{c c}
 \AxiomC{$\Gamma \Rightarrow A$}
 \AxiomC{$\Gamma, B \Rightarrow \Delta$}
 \RightLabel{$L \supset$} 
 \BinaryInfC{$\Gamma, A \supset B \Rightarrow \Delta$}
 \DisplayProof \;\;\;
 &
 \AxiomC{$ \Gamma, A \Rightarrow B$}
 \RightLabel{$R \supset$} 
 \UnaryInfC{$\Gamma \Rightarrow A \supset B$}
 \DisplayProof
 \\[3ex]
\end{tabular}
\end{center}

\begin{remark}\label{NablaDisOverDisjunction}
The modality $\nabla$ commutes with all finite disjunctions (including $\bot$), i.e., $\mathbf{STL} \vdash \nabla \bot \Leftrightarrow \bot$ and $\mathbf{STL} \vdash \nabla (A \vee B) \Leftrightarrow  \nabla A \vee \nabla B$. The proof is just a proof-theoretical presentation of the algebraic fact that any left adjoint preserves all joins. For $\mathbf{STL} \vdash \nabla \bot \Leftrightarrow \bot$, one direction is clear. For the other, consider the following proof-tree.
\begin{center}
  	\begin{tabular}{c}
  	
   		\AxiomC{$\bot \Rightarrow  $}
   		\RightLabel{\footnotesize$(Rw)$}
   		\UnaryInfC{$\bot \Rightarrow \top \to \bot$}
   		\RightLabel{\footnotesize$\nabla$}
   		\UnaryInfC{$\nabla \bot \Rightarrow \nabla (\top \to \bot)$}
   		
   		\AxiomC{$ \Rightarrow \top $}
   		\AxiomC{$\bot \Rightarrow  \bot$}
   		\RightLabel{\footnotesize$L \to$}
   		\BinaryInfC{$\nabla (\top \to \bot) \Rightarrow \bot$}
   		
   		\RightLabel{\footnotesize$cut$}
   		\BinaryInfC{$\nabla \bot \Rightarrow \bot$}
   		\DisplayProof\\[5ex]
	\end{tabular}
\end{center}
For $\mathbf{STL} \vdash \nabla (A \vee B) \Leftrightarrow  \nabla A \vee \nabla B$, again one direction is easy. For the other direction and for the sake of readability, we prove the apparently more general claim that if $\mathbf{STL} \vdash \nabla A \Rightarrow C$ and $\mathbf{STL} \vdash \nabla B \Rightarrow C$, then $\mathbf{STL} \vdash \nabla (A \vee B) \Rightarrow C$:
\begin{center}
  	\begin{tabular}{c}
  	
   		\AxiomC{$\nabla A \Rightarrow C $}
   		\RightLabel{\footnotesize$(Lw)$}
   		\UnaryInfC{$\nabla A, \top \Rightarrow C$}
   		\RightLabel{\footnotesize$(R \to)$}
   		\UnaryInfC{$A \Rightarrow \top \to C$}
   		
   		\AxiomC{$\nabla B \Rightarrow C $}
   		\RightLabel{\footnotesize$(Lw)$}
   		\UnaryInfC{$\nabla B, \top \Rightarrow C$}
   		\RightLabel{\footnotesize$(R \to)$}
   		\UnaryInfC{$B \Rightarrow \top \to C$}
   		
   		\RightLabel{\footnotesize$L \vee$}
   		\BinaryInfC{$A \vee B \Rightarrow \top \to C$}
   		\RightLabel{\footnotesize$\nabla $}
   		\UnaryInfC{$\nabla (A \vee B) \Rightarrow \nabla (\top \to C)$}
   		
   		\AxiomC{$ \Rightarrow \top $}
   		\AxiomC{$C \Rightarrow C$}
   		\RightLabel{\footnotesize$L \to$}
   		\BinaryInfC{$\nabla (\top \to C) \Rightarrow C$}
   		
   		\RightLabel{\footnotesize$cut$}
   		\BinaryInfC{$\nabla (A \vee B) \Rightarrow C$}
   		\DisplayProof\\[7ex]
	\end{tabular}
\end{center}
Then, the claim is provable by setting $C=\nabla A \vee \nabla B$.
\end{remark}

\begin{remark}\label{MakingOneFormula}
Over $\mathbf{STL}(N)$, any sequent $\Gamma \Rightarrow \Delta$ is equivalent to the simpler form $(\, \Rightarrow \bigwedge \Gamma \to \bigvee \Delta)$. One direction is clear. For the other, consider the following proof-tree:

\begin{center}
 \begin{tabular}{c c}
 \AxiomC{$\Rightarrow \bigwedge \Gamma \rightarrow \bigvee \Delta$}
\RightLabel{\footnotesize $N$}
\UnaryInfC{$ \Rightarrow \nabla(\bigwedge \Gamma \rightarrow \bigvee \Delta)$}
\AxiomC{$ $}
\doubleLine
\UnaryInfC{$ \Gamma \Rightarrow \bigwedge \Gamma$}
\AxiomC{$ $}
\doubleLine
\UnaryInfC{$\Gamma, \bigvee \Delta \Rightarrow \Delta$}
\RightLabel{\footnotesize $L\to$}
\BinaryInfC{$\Gamma, \nabla(\bigwedge \Gamma \to \bigvee \Delta) \Rightarrow  \Delta$}
\RightLabel{\footnotesize $cut$}
 \BinaryInfC{$\Gamma \Rightarrow \Delta$}
 \DisplayProof
 \\[3ex]
\end{tabular}
\end{center}
\normalsize where the double lines mean the existence of an omitted yet easy proof tree. 
\end{remark}

In the following lemma, we show that any of the rules in $\{D, N, R, L, Fa, Fu\}$ is equivalent to some axioms over $\mathbf{STL}$. 
\begin{lemma}\label{EquivProofSystems}
Consider the following list of axioms:
\begin{itemize}
\item[$\bullet$] $(D_a)$: $A \wedge (B \vee C) \Leftrightarrow (A \wedge B) \vee (A \wedge C)$.
\item[$\bullet$] $(N_a)$: $\nabla (A \wedge B) \Leftrightarrow \nabla A \wedge \nabla B$ and $\nabla \top \Leftrightarrow \top$.
\item[$\bullet$] $(R_a)$: $A \Rightarrow \nabla A$.
\item[$\bullet$] $(L_a)$: $\nabla A \Rightarrow A$.
\item[$\bullet$] $(Fa_a)$: $A \Leftrightarrow \nabla (\top \to A)$.
\item[$\bullet$] $(Fu_a)$: $A \Leftrightarrow \top \to \nabla A$.
\end{itemize}
and note that they are reminiscent of the algebraic conditions we had before. Then, for any $C \subseteq \{D, N, R, L, Fa, Fu\}$, we have $\mathbf{STL}(C)=\mathbf{STL}(C_a)$, where $C_a=\{\mathcal{R}_a \mid \mathcal{R} \in C\}$.
\end{lemma}
\begin{proof}
The equivalence of the rule $(D)$ and the axiom $(D_a)$ is an easy well-known fact.
For the rules $(N)$, $(R)$ and $(L)$, the proof of the equivalence is easy and can be found in \cite{ImSpace}.
For the case $(Fa)$, to prove the axioms, consider the following proof-trees: 
\begin{center}
 \begin{tabular}{c c}
 \AxiomC{$  \Rightarrow \top$}
 \AxiomC{$ A \Rightarrow A$}
 \RightLabel{\footnotesize $L \to$} 
 \BinaryInfC{$ \nabla (\top \rightarrow A) \Rightarrow A$}
 \DisplayProof
& \;\;\;\;\;
 \AxiomC{$ A \Rightarrow A$}
 \RightLabel{\footnotesize$Lw$} 
 \UnaryInfC{$ A, \top \Rightarrow A$}
 \RightLabel{\footnotesize$Fa$} 
 \UnaryInfC{$A \Rightarrow \nabla (\top \rightarrow A)$}
 \DisplayProof
 \\[3ex]
\end{tabular}
\end{center} 
Note that the axiom $\nabla (\top \to A) \Rightarrow A$ is even provable in $\mathbf{STL}$. To prove the rule $(Fa)$ from the axiom $(Fa)_a$, if we have $\Gamma, A \Rightarrow B$, then:
\begin{center}
	\begin{tabular}{c}
  	\AxiomC{$ $}
  	\RightLabel{\footnotesize $(Fa_a)$}
  	\UnaryInfC{$ \bigwedge \Gamma \Rightarrow \nabla ( \top \to \bigwedge \Gamma) $}
  	
  	\AxiomC{$ \Rightarrow \top $}
  	\AxiomC{$\bigwedge \Gamma \Rightarrow \bigwedge \Gamma $}
  	\RightLabel{\footnotesize $L \to$}
   		\BinaryInfC{$\nabla ( \top \to \bigwedge \Gamma) \Rightarrow \bigwedge \Gamma $}
   		
   		\AxiomC{$ \Gamma, A \Rightarrow B $}
   		\doubleLine
        \UnaryInfC{$\bigwedge \Gamma, A \Rightarrow B $}
   		\BinaryInfC{$\nabla ( \top \to \bigwedge \Gamma), A \Rightarrow B$}
   		 \RightLabel{$R\to$}
   		\UnaryInfC{$\top \to \bigwedge \Gamma \Rightarrow A \to B$}
   		\RightLabel{$\nabla$}
   		\UnaryInfC{$\nabla (\top \to \bigwedge \Gamma) \Rightarrow \nabla (A \to B)$}
   		 \RightLabel{$cut$}
   		\BinaryInfC{$\bigwedge \Gamma \Rightarrow \nabla (A \to B)$}
   		\doubleLine
   		\UnaryInfC{$\Gamma \Rightarrow \nabla(A \to B)$}
   		\DisplayProof
	\end{tabular}
\end{center}
\normalsize where the double lines mean the existence of an omitted yet easy proof tree. \\
For the case $(Fu)$, to prove the axioms, consider the following proof-trees: 
\begin{center}
 \begin{tabular}{c c}
\AxiomC{$ \nabla A \Rightarrow \nabla A$}
 \RightLabel{\footnotesize$Lw$} 
 \UnaryInfC{$\nabla A, \top \Rightarrow \nabla A$}
 \RightLabel{\footnotesize$R \to$} 
 \UnaryInfC{$A \Rightarrow \top \to \nabla A $}
 \DisplayProof
 & \;\;\;\;\;
 \AxiomC{$ \Rightarrow \top$}
 \AxiomC{$ \nabla A \Rightarrow \nabla A$}
 \RightLabel{\footnotesize$Fu$} 
 \BinaryInfC{$\top \rightarrow \nabla A \Rightarrow A$}
 \DisplayProof
\end{tabular}
\end{center} 
Note that $A \Rightarrow \top \to \nabla A$ is even provable in $\mathbf{STL}$. 
To prove the rule $(Fu)$ from the axiom $(Fu)_a$, we first prove the following claim:\\

\noindent \textbf{Claim.} $\nabla \Sigma \Rightarrow \nabla \Lambda \vdash_{\mathbf{STL}(Fu_a)} \Sigma \Rightarrow \Lambda$, for any sequent $\Sigma \Rightarrow \Lambda$.\\

\noindent \emph{Proof of the Claim.} Consider the following proof tree and note that by Remark \ref{NablaDisOverDisjunction}, $\nabla$ commutes with all disjunctions, i.e., $\mathbf{STL} \vdash \bigvee \nabla \Lambda \Rightarrow \nabla \bigvee \Lambda$:
\begin{center}
 \begin{tabular}{c}
 \AxiomC{$\nabla \Sigma \Rightarrow \nabla \Lambda$}
 \doubleLine
 \UnaryInfC{$\nabla \Sigma \Rightarrow \bigvee \nabla \Lambda$}
 \RightLabel{\footnotesize$cut$}
  \AxiomC{$\bigvee \nabla \Lambda \Rightarrow \nabla \bigvee \Lambda$} 
 \BinaryInfC{$\nabla \Sigma \Rightarrow \nabla \bigvee \Lambda$}
 \doubleLine
 \RightLabel{\footnotesize $L \to$} 
 \UnaryInfC{$ \{\nabla (\top \to \nabla \sigma)\}_{\sigma \in \Sigma} \Rightarrow \nabla \bigvee \Lambda$}
 \RightLabel{\footnotesize$R \to$} 
 \UnaryInfC{$\{\top \to \nabla \sigma\}_{\sigma \in \Sigma} \Rightarrow \top \to \nabla \bigvee \Lambda$}
 \doubleLine
\RightLabel{\footnotesize $\text{\emph{cut with some instances of $(Fu_a)$}}$}
 \UnaryInfC{$\{ \sigma\}_{\sigma \in \Sigma} \Rightarrow \bigvee \Lambda$}
  \doubleLine
 \UnaryInfC{$\Sigma \Rightarrow \Lambda$}
 \DisplayProof
\end{tabular}
\end{center}
where the double lines mean the existence of an omitted easy proof tree. More precisely, the doubleline with the label $L\to$ means applying the rule $L \to$ many times to change any $\nabla \sigma$ to $\nabla ( \top \to \nabla \sigma)$ and the doubleline with the label ``\emph{cut with some instances of $(Fu_a)$}" means using cut on both sides of the sequent with the instances $\sigma \Rightarrow \top \to \nabla \sigma$ and $\top \to \nabla \bigvee \Lambda \Rightarrow \bigvee \Lambda$. \qed\\

\noindent Using the claim, it is now easy to prove the rule $(Fu)$ from $(Fu_a)$:
 \begin{center}
 \begin{tabular}{c}
 \AxiomC{$\nabla \Gamma \Rightarrow A$}
 \AxiomC{$ \nabla \Gamma, B \Rightarrow \nabla \Delta$}
 \RightLabel{\footnotesize$L \to$} 
 \BinaryInfC{$\nabla \Gamma, \nabla (A \to B) \Rightarrow \nabla \Delta$}
 \doubleLine
 \RightLabel{$\text{\emph{the claim}}$} 
 \UnaryInfC{$\Gamma, A \to B \Rightarrow \Delta$}
 \DisplayProof
 \\[3ex]
\end{tabular}
\end{center} 
\end{proof}
It is clear that $\nabla$-algebras provide a natural algebraic semantics for the proof systems introduced in this section.
\begin{definition}\label{t4-1}\emph{(Algebraic Semantics)}
Let $\mathcal{A}=(\mathsf{A}, \nabla, \to)$ be a $\nabla$-algebra and $V:\mathcal{L} \to A$ be an assignment, mapping formulas to the elements of $A$. The pair $(\mathcal{A}, V)$ is called an \emph{algebraic model} if
$V(\top)=1$, $V(\bot)=0$,
$V(\nabla A)=\nabla V(A)$, and
$V(A \circ B)=V(A) \circ V(B)$, for any $\circ \in \{\wedge, \vee, \to , \supset\}$.
The case for $\supset$ only appears if we work with the language $\mathcal{L}_i$ and $\mathcal{A}$ is Heyting. We say $(\mathcal{A}, V) \vDash \Gamma \Rightarrow \Delta$ when $\bigwedge_{\gamma \in \Gamma} V(\gamma) \leq \bigvee_{\delta \in \Delta} V(\delta)$. Let $\mathfrak{C}$ be a class of $\nabla$-algebras. We write $\{\Gamma_i \Rightarrow \Delta_i\}_{i \in I} \vDash_{\mathfrak{C}} \Gamma \Rightarrow \Delta$, if for any $\mathcal{A} \in \mathfrak{C}$ and any $V$, having $(\mathcal{A}, V) \vDash \Gamma_i \Rightarrow \Delta_i$, for all $i\in I$, implies $(\mathcal{A}, V) \vDash \Gamma \Rightarrow \Delta$. 
\end{definition}
As it is expected, for any $C \subseteq \{N, H, D, R, L, Fa, Fu\}$, the system $ \mathbf{STL}(C)$ is sound and complete with respect to the corresponding variety $\mathcal{V}(C)$.
\begin{theorem}(Soundness-Completeness) \label{AlgebraicSoundComplete} Let $C \subseteq \{N, H, D, R, L, Fa, Fu\}$. Then, $\{\Gamma_i \Rightarrow \Delta_i\}_{i \in I} \vdash_{\mathbf{STL}(C)} \Gamma \Rightarrow \Delta$ iff $\{\Gamma_i \Rightarrow \Delta_i\}_{i \in I} \vDash_{\mathcal{V}(C)} \Gamma \Rightarrow \Delta$.
\end{theorem}
\begin{proof}
The proof for the system $\mathbf{STL}$ is routine and extensively explained in \cite{ImSpace}. For $\mathbf{STL}(H)$, the proof is similar. For all the other cases, the claim is a trivial consequence of the soundness-completeness theorem for $\mathbf{STL}$ and $\mathbf{STL}(H)$ combined with Lemma \ref{EquivProofSystems}. The reason is that over these two systems, each rule in $\{N, D, R, R, L, Fa, Fu\}$ is equivalent to a set of axioms, which are precisely those defining the corresponding variety of $\nabla$-algebras.
\end{proof}

Using the ring-theoretic representation for Heyting algebras, Theorem \ref{RingRepresentationForHeyting}, and normal distributive $\nabla$-algebras, Theorem \ref{RingRepresentation}, we can use Theorem \ref{AlgebraicSoundComplete} to prove the following ring-theoretic completeness result:

\begin{corollary}
\begin{description}
    \item[$\bullet$]
The logic $\mathsf{IPC}$ is sound and complete with respect to its algebraic interpretation in Heyting algebras in the form $\mathcal{RI}(R)$, where $R$ is a commutative unital ring.
     \item[$\bullet$]
Let $C \subseteq \{Fa, Fu\}$. Then, the logic $\mathbf{STL}(N, D, C)$ is sound and complete with respect to its algebraic interpretation in $\nabla$-algebras in the class $\{\mathfrak{Al}(\mathcal{{R}}) \mid \mathcal{R} \in \mathfrak{R}(C)\}$.    
\end{description}
\end{corollary}

Finally, let us turn to deductive interpolation. Recall that a system $G$ has \emph{deductive interpolation} whenever for any family of sequents $\mathcal{S} \cup \{S\}$, if $\mathcal{S} \vdash_G S$, then there exists a family $\mathcal{T}$ of sequents such that $V(\mathcal{T}) \subseteq V(\mathcal{S}) \cap V(S)$, $\mathcal{T} \vdash_G S$ and $\mathcal{S} \vdash_G T$, for any $T \in \mathcal{T}$. It is known that if the system $G$ is sound and complete with respect to a variety $\mathcal{V}$ and $\mathcal{V}$ has the amalgamation property, then $G$ enjoys the deductive interpolation property. For more information, see \cite{George}.

\begin{theorem}(Deductive Interpolation) 
Let $C \subseteq \{H, R, L, Fa\}$. Then, the sequent calculus $ \mathbf{STL}(N, D, C)$ has  deductive interpolation.
\end{theorem}
\begin{proof}
The claim is a consequence of Theorem \ref{Amalgamation}.
\end{proof}

\begin{remark}
In the definition of  deductive interpolation for $\mathbf{STL}(N, D, C)$, where $C \subseteq \{H, R, L, Fa\}$, if $\mathcal{S}$ is finite, one can always choose the set $\mathcal{T}$ to be a singleton in the form $\{\Rightarrow I\}$. To see why, first recall that by Remark \ref{MakingOneFormula}, over $\mathbf{STL}(N, D, C)$, any sequent $\Gamma \Rightarrow \Delta$ is equivalent to the sequent $(\, \Rightarrow \bigwedge \Gamma \to \bigvee \Delta)$. Now, to simplify $\mathcal{T}$, w.l.o.g, we can first assume that $\mathcal{T}=\{\Gamma_j \Rightarrow \Delta_j\}_{j \in J}$ is finite and then it is easy to see that the formula $I=\bigwedge_{j \in J}(\bigwedge \Gamma_j \to \bigvee \Delta_j)$ works. 
\end{remark}

\subsection{Kripke and Topological Semantics}

In this subsection, we will recall the topological and Kripke semantics for $\mathbf{STL}$ as introduced in \cite{ImSpace}. We will explore the relationship between these two semantics and the main algebraic semantics discussed earlier. Finally, we will extend the soundness-completeness results from \cite{ImSpace} to include the calculi with the new rules in $\{Fa, Fu\}$.

\subsubsection{Kripke Semantics}
The tuple $(W, \leq, R, V)$ is called a \emph{Kripke model} if $(W, \leq, R)$ is a Kripke frame and $V:At(\mathcal{L}) \to U(W, \leq)$ is a function, where $At(\mathcal{L})$ is the set of all the atoms of $\mathcal{L}$. For any Kripke model $(W, \leq, R, V)$, the forcing relation for the atomic formulas, conjunction and disjunction is defined as usual. For $\nabla$, $\to$ and $\supset$, we have:
\begin{itemize}
\item[$\bullet$]
$w \Vdash \nabla A$ iff there is $u \in W$ such that $(u, w) \in R$ and $u \Vdash A$,
\item[$\bullet$]
$w \Vdash A \to B$ iff for any $u \in W$ such that $(w, u) \in R$, if  $u \Vdash A$ then $u \Vdash B$,
\item[$\bullet$]
$w \Vdash A \supset B$ iff for any $u \in W$ such that $w \leq u $, if  $u \Vdash A$ then $u \Vdash B$.
\end{itemize}
We say $w \Vdash \Gamma \Rightarrow \Delta$ when $w \Vdash \bigwedge \Gamma$ implies $ w \Vdash \bigvee \Delta$. If $w \Vdash \Gamma \Rightarrow \Delta$, for all $w \in W$, we write $(W, \leq, R, V) \vDash \Gamma \Rightarrow \Delta$. Let $\mathfrak{C}$ be a class of Kripke frames. We write $\{\Gamma_i \Rightarrow \Delta_i\}_{i \in I} \vDash_{\mathfrak{C}} \Gamma \Rightarrow \Delta$, if for any $(W, \leq, R) \in \mathfrak{C}$ and any $V$, having $(W, \leq, R, V) \vDash \Gamma_i \Rightarrow \Delta_i$, for all $i \in I$, implies $(W, \leq, R, V) \vDash \Gamma \Rightarrow \Delta$.

Let $\mathcal{K}=(W, \leq, R)$ be a Kripke frame. Recall that by Theorem \ref{KripkeEmbedding} the tuple $\mathfrak{U}(\mathcal{K})=(U(W, \leq), \nabla_{\mathcal{K}}, \to_{\mathcal{K}})$ is a Heyting $\nabla$-algebra, where $\nabla_{\mathcal{K}}(U)=\{x \in X \mid \exists y \in U, (y, x) \in R \}$ and $U \to_{\mathcal{K}} V=\{x \in X \mid \forall y \in U, [(x, y) \in R \Rightarrow y \in V] \}$. Moreover, recall that for any $C \subseteq \{N, R, L, Fa, Fu\}$, if $\mathcal{K} \in \mathbf{K}(C)$ then $\mathfrak{U}(\mathcal{K}) \in \mathcal{V}(C, H)$. 

\begin{remark}\label{KripkeIsAlgebra}
To connect Kripke models to algebraic models, it is easy to show that if $(W, \leq, R, V)$ is a Kripke model, then $(\mathfrak{U}(W, \leq, R), \tilde{V})$ is an algebraic model, where $\tilde{V}$ is the extension of $V$ from the atoms of $\mathcal{L}$ to all formulas in $\mathcal{L}$ defined by $\Tilde{V}(A)=\{w \in W \mid w \Vdash A\}$. Moreover, note that
$(W, \leq, R, V) \vDash \Gamma \Rightarrow \Delta$ iff $(\mathfrak{U}(W, \leq, R), \tilde{V}) \vDash \Gamma \Rightarrow \Delta$, for any sequent $\Gamma \Rightarrow \Delta$. Therefore, Kripke models must be considered as a special family of algebraic models.
\end{remark}
\begin{theorem}\label{KripkeCompleteness}(Kripke Completeness) 
Let $C \subseteq \{N, R, L, Fa, Fu\}$. Then $  \{\Gamma_i \Rightarrow \Delta_i\}_{i \in I} \vdash_{\mathbf{STL}(C, D)} \Gamma \Rightarrow \Delta$ iff $\{\Gamma_i \Rightarrow \Delta_i\}_{i \in I} \vDash_{\mathbf{K}(C)} \Gamma \Rightarrow \Delta$. It also holds if we replace $\mathbf{STL}(C, D)$ by $\mathbf{STL}(C, H)$.
\end{theorem}
\begin{proof}
By Remark \ref{KripkeIsAlgebra}, the validity in any Kripke model $(\mathcal{K}, V)$ is equivalent to the validity in the algebraic model $(\mathfrak{U}(\mathcal{K}), \tilde{V})$. Moreover, by Theorem \ref{KripkeEmbedding}, the $\nabla$-algebra $\mathfrak{U}(\mathcal{K})$ is Heyting and the passage from $\mathcal{K}$ to $\mathfrak{U}(\mathcal{K})$ preserves any condition in $\{N, R, L, Fa, Fu\}$. Therefore, the soundness part is an easy consequence of the algebraic soundness, Theorem \ref{AlgebraicSoundComplete}. To prove the completeness, by Theorem \ref{KripkeEmbedding}, any distributive $\nabla$-algebra can be embedded in $\mathfrak{U}(\mathcal{K})$, for some Kripke frame $\mathcal{K}$ and if the $\nabla$-algebra satisfies any condition in $\{N, R, L, Fa, Fu\}$, then so does the Kripke frame $\mathcal{K}$. Moreover, recall that the embedding preserves the Heyting implication, if the $\nabla$-algebra is also Heyting. Since $(\mathcal{K}, V) \vDash \Gamma \Rightarrow \Delta$, iff $(\mathfrak{U}(\mathcal{K}), \tilde{V}) \vDash \Gamma \Rightarrow \Delta$, for any sequent $\Gamma \Rightarrow \Delta$, the completeness easily follows from the algebraic completeness, Theorem \ref{AlgebraicSoundComplete}.
\end{proof}

\begin{corollary}(Complete Lattices) Let $C \subseteq \{N, H, D, R, L, Fa, Fu\}$. Then $  \{\Gamma_i \Rightarrow \Delta_i\}_{i \in I} \vdash_{\mathbf{STL}(C)} \Gamma \Rightarrow \Delta$ iff $\{\Gamma_i \Rightarrow \Delta_i\}_{i \in I} \vDash_{\mathcal{CV}(C)} \Gamma \Rightarrow \Delta$, where $\mathcal{CV}(C)$ is the class of all complete $\nabla$-algebras in $\mathcal{V}(C)$.
\end{corollary}
\begin{proof}
Soundness is trivial. For completeness, if $D \in C$, use Theorem \ref{KripkeCompleteness}, as $\mathfrak{U}(\mathcal{K})$ is always complete, for any Kripke frame $\mathcal{K}$ and if $D \notin C$, use the Dedekind-McNeille completion, Theorem \ref{DedekindCompletionThm} to embed any $\nabla$-algebra in $\mathcal{V}(C)$ into a complete one in $\mathcal{CV}(C)$ and use the algebraic completeness, Theorem \ref{AlgebraicSoundComplete}.
\end{proof}

\subsubsection{Topological Semantics}
Recall from Example \ref{DTS} that the pair $(X, f)$ of a topological space $X$ and a continuous map $f: X \to X$ is called a dynamic topological system and  if $f$ is a topological embedding (resp. surjection), then the dynamic topological system is called faithful (resp. full). For $C \subseteq \{Fa, Fu\}$, by $\mathbf{T}(C)$, we mean the class of all dynamic topological systems satisfying the conditions in $C$.

Moreover, recall that for any dynamic topological system $(X, f)$, the tuple $\mathcal{T}(X, f)=(\mathcal{O}(X), f^{-1}, \to_f)$ is a normal Heyting $\nabla$-algebra, where $U \to_f V=f_*(U \supset V)$, $f^{-1} \dashv f_*$ and $U \supset V=int(U^c \cup V)$ is the Heyting implication of the locale $\mathcal{O}(X)$ and if the space $X$ is $T_0$ (resp. $T_D$), then the dynamic topological system $(X, f)$ is faithful (resp. full) iff the $\nabla$-algebra $(\mathcal{O}(X), f^{-1}, \to_f)$ is faithful (resp. full). 

\begin{definition}
A tuple $(X, f, V)$ is called a \emph{topological model} if $(X, f)$ is a dynamic topological system and $V:\mathcal{L} \to \mathcal{O}(X)$ is an assignment such that $(\mathcal{T}(X, f), V)$ is an algebraic model.
We say $(X, f, V) \vDash \Gamma \Rightarrow \Delta$ when $\bigcap_{\gamma \in \Gamma} V(\gamma) \subseteq \bigcup_{\delta \in \Delta} V(\delta)$. Let $\mathfrak{C}$ be a class of dynamic topological systems. Then, we write $\{\Gamma_i \Rightarrow \Delta_i\}_{i \in I} \vDash_{\mathfrak{C}} \Gamma \Rightarrow \Delta$, if for any $(X, f) \in \mathfrak{C}$ and any $V$, having $(X, f, V) \vDash \Gamma_i \Rightarrow \Delta_i$, for all $i \in I$, implies $(X, f, V) \vDash \Gamma \Rightarrow \Delta$. 
\end{definition}

It is clear that topological models form a subfamily of algebraic models and hence the topological semantics is a restricted form of the algebraic semantics. A similar connection holds between the topological semantics and the Kripke semantics as any normal Kripke model gives rise to a topological model. In the following, we will explain how. But first, we need a lemma.

\begin{lemma}\label{EmbeddingForUpsetTopology}
Let $(P, \leq_P)$ and $(Q, \leq_Q)$ be two posets and $f: P \to Q$ be an order-preserving function. Then, if $P$ and $Q$ are equipped with their upset topologies, then $f: P \to Q$ is a continuous map. Moreover, it is a topological embedding iff it is an order-embedding. 
\end{lemma}
\begin{proof}
Continuity is easy as the inverse image of an upset is also an upset. For the rest, as any upset topology is $T_0$, by Theorem \ref{InjSurjforContinuous}, the map $f$ is a topological embedding iff $f^{-1}$ is surjective on the opens. Therefore, it is enough to show that $f$ is an order-embedding iff $f^{-1}: U(Q, \leq_Q) \to (P,\leq_P)$ is surjective. For the first direction, suppose that $f$ is an order-embedding. We claim that $f^{-1}(\uparrow \!\!f[U])=U$, for any upset $U$ in $P$. It is clear that this claim proves the surjectivity of $f^{-1}: U(Q, \leq_Q) \to (P,\leq_P)$. To prove the claim, the direction $U \subseteq f^{-1}(\uparrow \!\!f[U])$ is trivial. For the converse, if $x \in f^{-1}(\uparrow \!\!f[U])$, then by definition, there is $y \in U$ such that $f(y) \leq_Q f(x)$. Since $f$ is an order-embedding, $y \leq_P x$ which implies $x \in U$. Therefore,  $f^{-1}(\uparrow \!\!f[U])=U$. For the other direction,
assume that $f^{-1}: U(Q, \leq_Q) \to (P,\leq_P)$ is surjective, $f(x) \leq_Q f(y)$ and $x \nleq_P y$. Then, set $U=\uparrow \!x$ and note that $x \in U$ but $y \notin U$. Since $f^{-1}$ is surjective over the upsets, there exists an upset $V \subseteq Q$ such that $f^{-1}(V)=U$. Since $x \in U$, we have $f(x) \in V$. Since $V$ is an upset and $f(x) \leq_Q f(y)$, we have $f(y) \in V$ which implies the contradictory consequence that $y \in f^{-1}(V)=U$. Hence, $f$ is an order-embedding. 
\end{proof}
Now, we show that any (faithful or full) normal Kripke frame gives rise to a (faithful or full) dynamic topological system with the same associated $\nabla$-algebra. 
Let $\mathcal{K}=(W, \leq, R)$ be a normal Kripke frame with the normality witness $\pi: W \to W$. Then, considering $W$ as a topological space with the upset topology, $(W, \pi)$ is a dynamic topological system, as $\pi$ is order-preserving and hence continuous, by Lemma \ref{EmbeddingForUpsetTopology}. Using Lemma \ref{NormalForConditions} and Lemma \ref{EmbeddingForUpsetTopology}, the tuple $\mathcal{K}=(W, \leq, R)$ is a full (resp. faithful) Kripke frame iff $(W, \pi)$ is a full (resp. faithful) dynamic topological system. Moreover, the associated $\nabla$-algebra to a Kripke frame and its corresponding dynamic topological system are the same, i.e., $\mathfrak{U}(W, \leq, R)=\mathcal{T}(W, \pi)$. To prove, as $\pi$ is the normality witness, we have $(x, y) \in R$ iff $x \leq \pi(y)$. Thus, 
\[
\nabla_{\mathcal{K}}(U)=\{x \in W \mid \exists y \in U, (y, x) \in R\}=\pi^{-1}(U),
\]
for any upset $U$. Therefore, the $\nabla$'s of the $\nabla$-algebras are the same. Hence, using the uniqueness of the adjoints, we can see that $U \to_{\mathcal{K}} V$ and $U \to_{\pi} V$ are also equal, for any two upsets $U$ and $V$. 

The above discussion shows that if $(W, \leq, R, V)$ is a normal Kripke model with the normality witness $\pi$, then $(W, \pi, \tilde{V})$ is a topological model, where $\Tilde{V}(A)=\{w \in W \mid w \Vdash A\}$ and
$(W, \leq, R, V) \vDash \Gamma \Rightarrow \Delta$ iff $(W, \pi, \tilde{V}) \vDash \Gamma \Rightarrow \Delta$, for any sequent $\Gamma \Rightarrow \Delta$. Therefore:

\begin{theorem}(Topological completeness) Let $C \subseteq \{Fa, Fu\}$. Then $  \{\Gamma_i \Rightarrow \Delta_i\}_{i \in I} \vdash_{\mathbf{STL}(C, N, D)} \Gamma \Rightarrow \Delta$ iff $\{\Gamma_i \Rightarrow \Delta_i\}_{i \in I} \vDash_{\mathbf{T}(C)} \Gamma \Rightarrow \Delta$. It also holds if we replace $\mathbf{STL}(C, N, D)$ by $\mathbf{STL}(C, N, H)$.
\end{theorem}
\begin{proof}
The validity in a topological model $(X, f, V)$ is the same as the validity in the algebraic model $(\mathcal{T}(X, f), V)$. Moreover,  by Example \ref{DTS}, the $\nabla$-algebra $\mathcal{T}(X, f)$ is Heyting and the passage from $(X, f)$ to $\mathcal{T}(X, f)$ preserves any condition in $\{Fa, Fu\}$. Therefore, the soundness part is an easy consequence of the algebraic soundness, Theorem \ref{AlgebraicSoundComplete}. For completeness, use the Kripke completeness, Theorem \ref{KripkeCompleteness}, and the observation that any normal Kripke model gives rise to an equivalent topological model and this passage preserves the conditions in $\{Fa, Fu\}$.
\end{proof}

\bibliographystyle{plain}
\bibliography{Proof}

\end{document}